\documentclass[12pt,a4paper]{amsart}

\usepackage[utf8]{inputenc}
\usepackage[T1]{fontenc}
\usepackage[normalem]{ulem}
\usepackage{cancel}
\usepackage{xcolor}
\usepackage{url}

\usepackage[pdftex]{graphicx}
\usepackage{wrapfig}
\usepackage{latexsym}
\usepackage{amsmath,amssymb,amsthm}
\usepackage[shortlabels]{enumitem}   

\usepackage{mathpazo} 
\usepackage{nicefrac} 
\usepackage{mathtools} 
\allowdisplaybreaks 

\usepackage{scalerel,stackengine}
\stackMath
\newcommand\reallywidehat[1]{%
\savestack{\tmpbox}{\stretchto{%
  \scaleto{%
    \scalerel*[\widthof{\ensuremath{#1}}]{\kern-.6pt\bigwedge\kern-.6pt}%
    {\rule[-\textheight/2]{1ex}{\textheight}}
  }{\textheight}%
}{0.5ex}}%
\stackon[1pt]{#1}{\tmpbox}
}

\newtheorem{Theorem}{Theorem}
\newtheorem{Lemma}[Theorem]{Lemma}		
\newtheorem{Proposition}[Theorem]{Proposition}		
\newtheorem{Corollary}[Theorem]{Corollary}	
\theoremstyle{definition}
\newtheorem{Remark}[Theorem]{Remark}	 
\newtheorem{Definition}[Theorem]{Definition} 
   
\newcommand{\C}{\mathbb{C}} 
\newcommand{\R}{\mathbb{R}} 
\newcommand{\Z}{\mathbb{Z}} 
\newcommand{\N}{\mathbb{N}} 

\newcommand{\RT}{\operatorname{Re}}
\newcommand{\IT}{\operatorname{Im}}
\newcommand{\iu}{\mathrm{i}}
\newcommand{\eu}{\mathrm{e}}

\DeclareMathOperator{\sign}{sign}
\DeclareMathOperator{\spann}{span}
\DeclareMathOperator{\range}{range}

\usepackage[colorlinks=true,urlcolor=blue]{hyperref} 			

\begin{document}

\title[Global continua of solutions to the Lugiato-Lefever model]{Global continua of solutions to the Lugiato-Lefever model for frequency combs obtained by two-mode pumping}

\author{Elias Gasmi}
\address{E. Gasmi\hfill\break 
Institute for Analysis,\hfill\break
Karlsruhe Institute of Technology (KIT), \hfill\break
D-76128 Karlsruhe, Germany}
\email{elias.gasmi@kit.edu}

\author{Tobias Jahnke}
\address{T. Jahnke\hfill\break 
Institute for Applied and Numerical Mathematics,\hfill\break
Karlsruhe Institute of Technology (KIT), \hfill\break
D-76128 Karlsruhe, Germany}
\email{tobias.jahnke@kit.edu}

\author{Michael Kirn}
\address{M. Kirn\hfill\break 
Institute for Applied and Numerical Mathematics,\hfill\break
Karlsruhe Institute of Technology (KIT), \hfill\break
D-76128 Karlsruhe, Germany}
\email{michael.kirn@kit.edu}

\author{Wolfgang Reichel}
\address{W. Reichel \hfill\break 
Institute for Analysis,\hfill\break
Karlsruhe Institute of Technology (KIT), \hfill\break
D-76128 Karlsruhe, Germany}
\email{wolfgang.reichel@kit.edu}

\date{\today} 
	



\begin{abstract} 
We consider Kerr frequency combs in a dual-pumped microresonator as time-periodic and spatially $2\pi$-periodic traveling wave solutions of a variant of the Lugiato-Lefever equation, which is a damped, detuned and driven nonlinear Schr\"odinger equation given by $\mathrm{i}a_\tau =(\zeta-\mathrm{i})a- d a_{x x}-|a|^2a+\mathrm{i}f_0+\mathrm{i}f_1\mathrm{e}^{\mathrm{i}(k_1 x-\nu_1 \tau)}$. The main new feature of the problem is the specific form of the source term $f_0+f_1\mathrm{e}^{\mathrm{i}(k_1 x-\nu_1 \tau)}$ which describes the simultaneous pumping of two different modes with mode indices $k_0=0$ and $k_1\in \N$. We prove existence and uniqueness theorems for these traveling waves based on a-priori bounds and fixed point theorems. Moreover, by using the implicit function theorem and bifurcation theory, we show how non-degenerate solutions from the $1$-mode case, i.e. $f_1=0$, can be continued into the range $f_1\not =0$. Our analytical findings apply both for anomalous ($d>0$) and normal ($d<0$) dispersion, and they are illustrated by numerical simulations.
\end{abstract}

\maketitle

\section{Introduction} 

Optical frequency comb devices are extremely promising in many applications such as, e.g., optical frequency metrology \cite{Udem2002}, spectroscopy \cite{picque2019frequency,yang2017microresonator}, ultrafast optical ranging \cite{trocha2018ultrafast}, and high capacity optical communications \cite{marin2017microresonator}. For many of these applications the Kerr soliton combs are generated by using a monochromatic pump. However, recently new pump schemes have been discussed, where more than one resonator mode is pumped, cf. \cite{Taheri_2017}. The pumping of two modes can have a number of important advantages. In particular, $1$-solitons arising from a dual-pump scheme can be spectrally broader and spatially more localized than $1$-solitons arising from a monochromatic pump, cf. \cite{Gasmi_Peng_Koos_Reichel} for a comprehensive discussion of the theoretical advantages. Mathematically, Kerr comb dynamics are described by the Lugiato-Lefever equation (LLE), a damped, driven and detuned nonlinear Schr\"odinger equation \cite{Godey_et_al2014,Lugiato_Lefever1987,Parra-Rivas2018}. Our analysis relies on a variant of the LLE which is modified for two-mode pumping, cf. \cite{Taheri_2017} and \cite{Gasmi_Peng_Koos_Reichel} for a derivation. Using dimensionless, normalized quantities this equation takes the form
\begin{equation}\label{PDE}
\mathrm{i}a_\tau =(\zeta-\mathrm{i})a- d a_{x x}-|a|^2a+\mathrm{i}f_0+\mathrm{i}f_1\mathrm{e}^{\mathrm{i}(k_1 x-\nu_1 \tau)}, \qquad a \ 2\pi\text{-periodic in } x.
\end{equation}
Here, $a(\tau,x)$ represents the optical intracavity field as a function of normalized time $\tau=\frac{\kappa}{2}t$ and angular position $x\in[0,2\pi]$ within the ring resonator. The constant $\kappa>0$ describes the cavity decay rate and $d=\frac{2}{\kappa}d_2$ quantifies the dispersion in the system (where $\omega_k = \omega_0+d_1k+d_2k^2$ is the cavity dispersion relation between the resonant frequencies $\omega_k$ and the relative indices $k\in\Z$). Here, the case $d<0$ amounts to normal and the case $d>0$ to anomalous dispersion. The resonant modes in the cavity are numbered by $k\in \Z$ with $k_0=0$ being the first and $k_1\in\N$ the second pumped mode. With $f_0,f_1$ we describe the normalized power of the two input pumps and $\omega_{p_0},\omega_{p_1}$ denote the frequencies of the two pumps. Since there are now two pumped modes there are also two normalized detuning parameters denoted by $\zeta=\frac{2}{\kappa}(\omega_0-\omega_{p_0})$ and $\zeta_1=\frac{2}{\kappa}(\omega_{k_1}-\omega_{p_1})$. They describe the offsets of the input pump frequencies $\omega_{p_0}$ and $\omega_{p_1}$ to the closest resonance frequency $\omega_0$ and $\omega_{k_1}$ of the microresonator. The particular form of the pump term $\iu f_0+\iu f_1\eu^{\iu(k_1x-\nu_1 \tau)}$ with $\nu_1=\zeta-\zeta_1+d k_1^2$ suggests to change into a moving coordinate frame and to study solutions of \eqref{PDE} of the form $a(\tau,x)=u(s)$ with $s=x-\omega \tau$ and $\omega=\frac{\nu_1}{k_1}$. These traveling wave solutions propagate with speed $\omega$ in the resonator and their profiles $u$ solve the ordinary differential equation
\begin{equation}\label{TME}
-d u''+\mathrm{i} \omega u'+(\zeta-\mathrm{i})u-|u|^2 u+\mathrm{i}f_0+\mathrm{i}f_1 \mathrm{e}^{\mathrm{i}k_1 s}=0, \qquad u \ 2\pi\text{-periodic}.
\end{equation}
In the case $f_1=0$ equation \eqref{PDE} amounts to the case of pumping only one mode. This case has been thoroughly studied, e.g. in \cite{DelHara_periodic,Gaertner_et_al,Godey_2017,Godey_et_al2014,Mandel,Miyaji_Ohnishi_Tsutsumi2010,Parra-Rivas2018,Parra-Rivas2014,Parra-Rivas2016,Perinet,Stanislavova_Stefanov}. In this paper we are interested in the case $f_1\neq 0$. Since the specific form of the forcing term is not essential for many of our results, we allow in the following for more general forcing terms 
$$f(s)=f_0+ f_1 e(s)$$
with a $2\pi$-periodic (not necessarily continuous) function $e:\R\to \C$ and $f_0, f_1\in \R$.
Hence, we consider the LLE
\begin{equation}\label{TWE}
-d u''+\mathrm{i} \omega u'+(\zeta-\mathrm{i})u-|u|^2 u+\mathrm{i}f(s)=0, \qquad u \ 2\pi\text{-periodic}.
\end{equation}
Our main results on the existence of solutions to \eqref{TWE} are stated in Section~\ref{sec:main_results}. In Section~\ref{sec:numerical} we illustrate our main analytical results by numerical simulations. The proofs of the main results are given in Section~\ref{sec:a-priori} (a-priori bounds), Section~\ref{sec:existence_and_uniqueness} (existence and uniqueness), and Section~\ref{sec:continuation} (continuation results). The appendix contains a technical result and a consideration of the case where in \eqref{TME} the value $k_1$ is not an integer but close to an integer.

\section{Main Results} \label{sec:main_results}

In the following we state our main results.
\begin{itemize}
\item Theorem~\ref{Existence} provides existence of at least one solution of \eqref{TWE} for any choice of the parameters and any choice of $f$.
\item Theorem~\ref{Hauptresultat} and Corollary~\ref{Korollar} describe how trivial (constant) solutions from the special case $f_1=0$ can be continued into non-trivial solutions for $f_1\not =0$.
\item Theorem~\ref{Fortsetzung_nichttrivial} and Corollary~\ref{Fortsetzung_nichttrivial_Korollar} show how a non-trivial solution from the case $f_1=0$ can be continued to $f_1\not=0$.
\end{itemize}
Our first theorem, which ensures the existence of a solution of \eqref{TWE} in the general case where $f_1$ does not need to vanish, is based on a-priori bounds and a variant of Schauder's fixed point theorem known as Schaefer's fixed point theorem. A corresponding uniqueness result, which applies whenever $|\zeta|\gg 1$ is sufficiently large or (essentially) $\|f\|_2\ll 1$ is sufficiently small is given in Theorem~\ref{Uniqueness} in Section~\ref{sec:existence_and_uniqueness} together with more precise details.

We will use the following Sobolev spaces. For $k\in\N$ the space $H^k(0,2\pi)$ consists of all square-integrable 
functions on $(0,2\pi)$ whose weak derivatives up to order $k$ exist and are square-integrable on $(0,2\pi)$. By $H^k_{\text{per}}(0,2\pi)$ we denote all locally square-integrable $2\pi$-periodic functions on $\R$ whose weak derivatives up to order $k$ exist and are locally square-integrable on $\R$. In both spaces the norm is given by $\|u\|= \bigl(\sum_{j=0}^k \|(\frac{d}{ds})^j u\|_{L^2(0,2\pi)}^2 \bigr)^{1/2}$. Clearly $H^k_{\text{per}}(0,2\pi)$ is a proper subspace of $H^k(0,2\pi)$ since $u\in H^k_{\text{per}}(0,2\pi)$ implies that $(\frac{d}{ds})^{j}u(0)=(\frac{d}{ds})^{j}u(2\pi)$ for $j=0,\ldots,k-1$. Unless otherwise stated, all of the above Hilbert spaces are spaces of complex valued functions over the field $\R$. In particular, for $v,w\in L^2(0,2\pi)$ we use the inner product $\langle v,w\rangle_2\coloneqq\RT\int_0^{2\pi} v\overline w\,ds$. The induced norm is denoted by $\| \cdot \|_2$. 

\begin{Theorem}\label{Existence}
Equation \eqref{TWE} has at least one solution $u\in H^2_{\text{per}}(0,2\pi)$ for any choice of the parameters $d\in \R\setminus\{0\}$, $\zeta,\omega \in \R$ and any choice of $f\in H^2(0,2\pi)$. 
\end{Theorem}

Next we address the question whether a known solution $u_0$ of \eqref{TWE} for $f_1=0$ can be continued into the regime $f_1\not =0$. This continuation will be done differently depending on whether $u_0$ is constant (trivial) or non-constant (non-trivial). Moreover, we first concentrate on one-sided continuations for $f_1>0$ (or $f_1<0$). Two-sided continuations will be discussed in Section~\ref{sec:two_sided}.

\subsection{One-sided continuation of trivial solutions}

In the special case $f_1=0$ there are trivial (constant) solutions $u_0 \in \C$ of \eqref{TWE} satisfying the algebraic equation 
\begin{equation} \label{trivial}
(\zeta-\mathrm{i})u_0-|u_0|^2 u_0+\mathrm{i}f_0=0.
\end{equation}
From \cite[Lemma 2.1]{Mandel} we know that for given $f_0 \in \R$ the curve of constant solutions can be parameterized by 
\begin{equation} \label{parametrization}
\zeta(t)=(1-t^2)f_0^2+\frac{t}{\sqrt{1-t^2}}, \quad u_0(t)=(1-t^2)f_0-\mathrm{i} f_0 t \sqrt{1-t^2}, \quad t\in(-1,1).
\end{equation}
In Figure~\ref{fig:trivial} we show the curve of the squared $L^2$-norm of all constant solutions of \eqref{TWE} for $f_1=0$ and $f_0=1$, $f_0=\frac{2\sqrt{2}}{\sqrt[4]{27}}$ and $f_0=2$. The curve may or may not have turning points which are characterized by $\zeta'(t)=0$. This condition can be formulated independently of $t$ by the equivalent condition $\zeta^2-4|u_0|^2 \zeta+1+3 |u_0|^4=0$. 
By a straightforward analysis one can show that with $f^*=\frac{2\sqrt{2}}{\sqrt[4]{27}}$ we have 
\begin{itemize}
\item no turning point for $|f_0|<f^*$ (cf. Figure~\ref{fig:trivial} green curve),
\item exactly one (degenerate) turning point for $|f_0|=f^*$ (cf. Figure~\ref{fig:trivial} red curve),
\item exactly two turning points for $|f_0|>f^*$ (cf. Figure~\ref{fig:trivial} blue curve).
\end{itemize}
\begin{wrapfigure}[14]{r}{0.65\textwidth}
\centering
\includegraphics[width=0.5\textwidth]{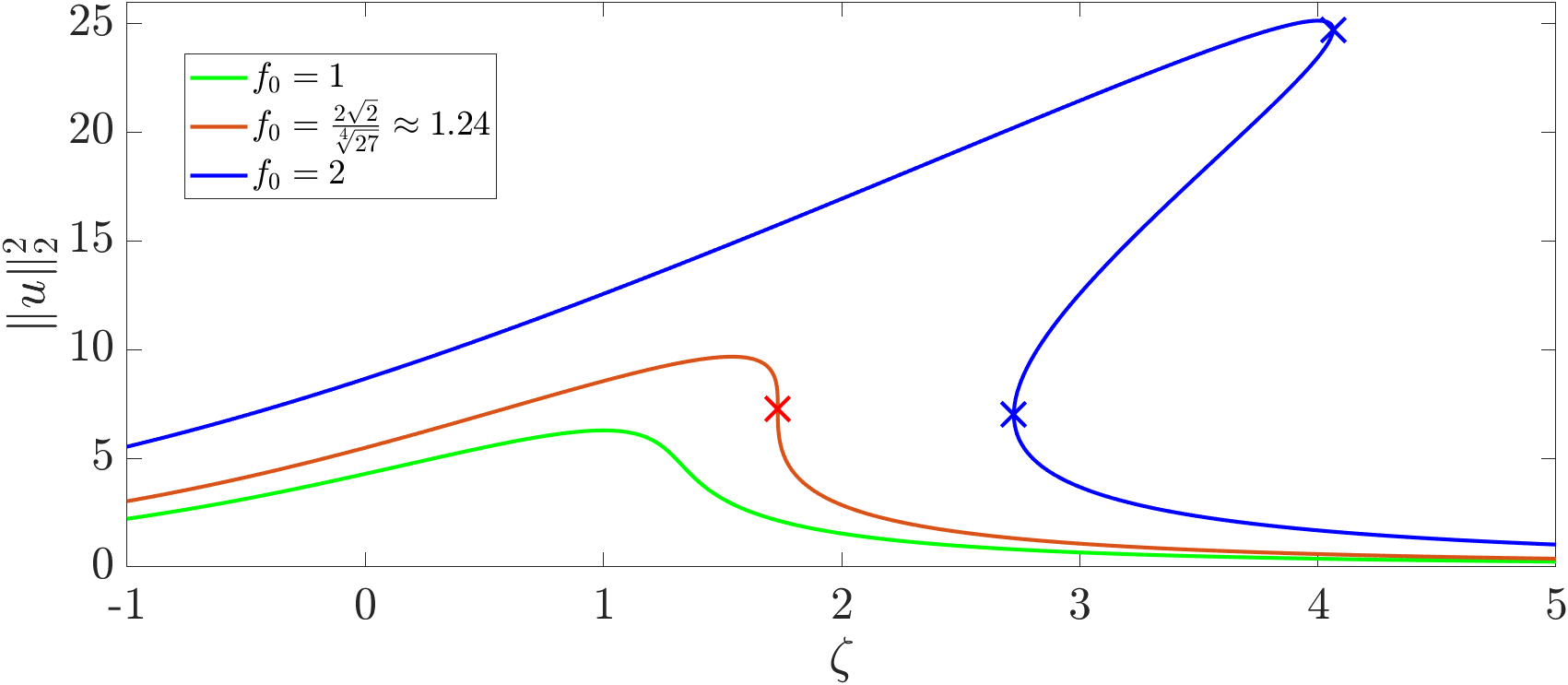} 
\caption{Curve of squared $L^2$-norm of all constant solutions of \eqref{TWE} for $f_1=0$ and $f_0=1$ (green), $f_0=\frac{2\sqrt{2}}{\sqrt[4]{27}}$ (red) and $f_0=2$ (blue) when $\zeta \in [-1,5]$. Turning points (if they exist) are marked with a cross.}
\label{fig:trivial}
\end{wrapfigure}

\noindent Note that for $|f_0|>f^*$, as a consequence of the existence of two turning points, three different constant solutions exist for certain values of $\zeta$. 

Starting from $f_1=0$ we use a kind of global implicit function theorem to continue a constant solution $u_0 \in \C$ of \eqref{TWE} with respect to $f_1$. This procedure is analyzed in Theorem~\ref{Hauptresultat}. The continuation works if the constant solution $u_0\in \C$ is non-degenerate in the following sense.

\begin{Definition} A solution $u\in H^2_{\text{per}}(0,2\pi)$ of \eqref{TWE} for $f_1=0$ is called non-degenerate if the kernel of the linearized operator 
$$
L_u\varphi \coloneqq -d \varphi'' +\iu\omega \varphi' +(\zeta-\iu-2|u|^2)\varphi-u^2\overline\varphi,  \quad \varphi\in H^2_{\text{per}}(0,2\pi)
$$
consists only of $\spann\{u'\}$.
\label{non_degenerate_solution}
\end{Definition}

\begin{Remark} \label{fredholm_etc} Note that $L_u: H^2_{\text{per}}(0,2\pi) \to L^2(0,2\pi)$ is a compact perturbation of the isomorphism $-d\frac{d^2}{dx^2}+\sign(d): H^2_{\text{per}}(0,2\pi) \to L^2(0,2\pi)$ and hence an index-zero Fredholm operator. Notice also that $\spann\{u'\}$ always belongs to the kernel of $L_u$. Non-degeneracy means that except for the obvious candidate $u'$ (and its real multiples) there is no other element of the kernel of $L_u$. Notice also that a constant solution $u_0$ is non-degenerate if the linearized operator $L_{u_0}$ is injective, and, as a consequence, invertible in suitable spaces.   
\end{Remark}

\begin{Lemma} A trivial solution $u_0\in\C$ of \eqref{TWE} for $f_1=0$ is non-degenerate if and only if
\begin{itemize}
\item[(a)] Case $\omega\not=0$: 
$$
\zeta^2-4|u_0|^2 \zeta+1+3 |u_0|^4 \neq 0.
$$
\item[(b)] Case $\omega=0$: 
$$
(\zeta+d m^2)^2-4|u_0|^2(\zeta+d m^2)+1+3|u_0|^4\neq 0 \quad \text{ for all } m\in \N_0.
$$
\end{itemize}
\label{charactrization_nondeg}
\end{Lemma}

\begin{proof} Let $\varphi\in H^2_{\text{per}}(0,2\pi)$ be in the kernel of the linearized operator, i.e., 
$$
-d \varphi'' +\iu\omega \varphi' +(\zeta-\iu-2|u_0|^2)\varphi-u_0^2\overline\varphi=0. 
$$
This implies that the Fourier coefficients $\varphi_m$ of the Fourier series
$\varphi=\sum_{m\in\Z} \varphi_m \mathrm{e}^{\mathrm{i}ms}$ have the property that
$$
(d m^2-\omega m+\zeta-\mathrm{i}-2|u_0|^2)\varphi_m -u_0^2 \overline{\varphi_{-m}}=0
$$
for all $m\in \Z$. If we also write down the complex conjugate of this equation
$$
 -\overline{u_0}^2 \varphi_{m}+(d m^2+\omega m+\zeta+\mathrm{i}-2|u_0|^2)\overline{\varphi_{-m}}=0
$$
then we see that non-degeneracy of $u_0$ is equivalent to the non-vanishing of the determinant for this two-by-two system in the variables $\varphi_m, \overline{\varphi_{-m}}$ for all $m\in\N_0$. Computing the determinant we obtain the condition
\begin{equation} \label{non_degenerate}
(\zeta+d m^2)^2-4|u_0|^2(\zeta+dm^2)+1+3|u_0|^4-\omega^2 m^2 -2\mathrm{i} \omega m\neq 0 \text{ for all } m\in\N_0.
\end{equation}
In the case $\omega\neq 0$ this is trivially satisfied for all $m\neq 0$ (because then the imaginary part is non-zero) and for $m=0$ by assumption (a) of the lemma. In the case $\omega=0$ condition \eqref{non_degenerate} can only be guaranteed by assumption (b).
\end{proof}

\begin{Remark} Trivial solutions of \eqref{TWE} for $f_1=0$ are determined by \eqref{trivial}. For $\omega\not=0$ all trivial solutions $u_0$ of \eqref{TWE} for $f_1=0$ are non-degenerate except those at the turning points described above. In the case $\omega=0$ all trivial solutions $u_0$ of \eqref{TWE} for $f_1=0$ are non-degenerate except those at the (potential) bifurcation points and the turning points. This is true (up to additional conditions ensuring transversality and simplicity of kernels) because the necessary condition for bifurcation w.r.t. $\zeta$ from the curve of trivial solutions is fulfilled if and only if the expression in (b) vanishes for at least one $m\in \N$, cf. \cite{Gaertner_et_al},\cite{Mandel}. 
\end{Remark}

\begin{Theorem}\label{Hauptresultat}
Let $d \in \R\setminus\{0\}$, $\zeta,\omega,f_0\in \R$ and $e \in H^2(0,2\pi)$ be fixed. Let furthermore $u_0 \in\C$ be a constant non-degenerate solution of \eqref{TWE} for $f_1=0$. Then the maximal continuum\footnote{A continuum is a closed and connected set.} $\mathcal{C}^+\subset [0,\infty) \times H^2_{\text{per}}(0,2\pi)$ of solutions $(f_1,u)$ of \eqref{TWE} with $(0,u_0)\in \mathcal{C}^+$ has the following properties:
\begin{itemize}
\item[(i)] locally near $(0,u_0)$ the set $\mathcal{C}^+$ is the graph of a smooth curve $f_1\mapsto (f_1,u(f_1))$,
\item[(ii)] $\mathcal{C}^+\cap [0,M]\times H^2_{\text{per}}(0,2\pi)$ is bounded for any $M>0$.
\end{itemize} 
Moreover, if $\operatorname{pr}_1(\mathcal{C}^+)$ denotes the projection of $\mathcal{C}^+$ onto the $f_1$-parameter component, then at least one of the following properties hold: 
\begin{enumerate}[(a)]
\item $\operatorname{pr}_1(\mathcal{C}^+)= [0,\infty)$,
\end{enumerate}
or 
\begin{enumerate}[(b)]
\item $\exists  u_0^+ \not = u_0: \,  (0,u_0^+) \in \mathcal{C}^+.$
\end{enumerate}
A maximal continuum  $\mathcal{C}^-\subset (-\infty,0]\times H^2_{\text{per}}(0,2\pi)$ with corresponding properties also exists.
\end{Theorem}

\begin{Remark} \label{Remark_interpretation}
If property (a) of Theorem~\ref{Hauptresultat} holds, then $\mathcal{C}^+$ is unbounded in the direction of the parameter $f_1\in [0,\infty)$ and hence this is an existence result for all $f_1 \in [0,\infty)$. Property (b) means that the continuum $\mathcal{C}^+$ returns to the $f_1=0$ line at a point $u_0^+ \not = u_0$.
\end{Remark}

\begin{Corollary}\label{Korollar} Property (a) in Theorem~\ref{Hauptresultat} holds in any of the following three cases,
\begin{itemize}
\item[(i)] $\displaystyle \sign(d)\zeta <-C(d,f_0)^2 \mathbf{1}_{d<0}-27\biggl(1+\frac{\pi f_0^2 |\omega|}{|d|}+\frac{\pi^2 f_0^4}{|d|}\biggr) C(d,f_0)^6$,
\item[(ii)] $\displaystyle \sign(d)\zeta >3C(d,f_0)^2+\frac{\omega^2}{4|d|}$,
\item[(iii)] $\displaystyle  \sqrt{3} C(d,f_0)<1$,
\end{itemize}
where
$$ 
C(d,f_0) =|f_0|(1+2\pi^2 f_0^2 |d|^{-1}). 
$$
In particular $|\zeta| \gg 1$ or $|f_0| \ll 1$ is sufficient. 
\end{Corollary}

\subsection{One-sided continuation of non-trivial solutions} 

One can ask the question whether also non-trivial (non-constant) solutions at $f_1=0$ may be continued into the regime of $f_1>0$. This depends on two issues: existence and non-degeneracy of a non-trivial solution of \eqref{TWE} for $f_1=0$. First we note that for $\omega=0$ there is a plethora of non-trivial solutions, cf. \cite{Gaertner_et_al},\cite{Mandel}. For $\omega\not =0$ we do not know whether non-trivial solutions exist for $f_1=0$. The fact that for $\omega\not =0$ there are no bifurcations from the curve of trivial solutions indicates that there may be no solutions other than the trivial ones. Although by the current state of understanding the hypotheses of Theorem~\ref{Fortsetzung_nichttrivial} (see below) can only be fulfilled for $\omega=0$, we allow in the following for general $\omega\in\R$.

\medskip

In order to describe the continuation from a non-degenerate non-trivial solution, let us first state some properties of \eqref{TWE} for $f_1=0$: if $u_0$ solves \eqref{TWE} for $f_1=0$ and if we denote its shifts by $u_\sigma(s)\coloneqq u_0(s-\sigma)$, then $u_\sigma$ also solves \eqref{TWE} for $f_1=0$. Hence 
$$
S:\left\{\begin{array}{rcl}
 \R & \to & \R \times H^2_\text{per}(0,2\pi), \vspace{\jot} \\
\sigma & \mapsto & (0,u_{\sigma})
\end{array}
\right.
$$
describes a trivial curve of solutions of \eqref{TWE} from which we wish to bifurcate at some point $(0,u_{\sigma_0})$. Recall also from non-degeneracy that $\ker L_{u_{\sigma}}=\spann\{u_\sigma'\}$. Since $L_{u_\sigma}^*$ also has a one-dimensional kernel, there exists $\phi_\sigma^*\in H^2_\text{per}(0,2\pi)$ such that $\ker L_{u_\sigma}^*=\spann\{\phi_\sigma^*\}$. Notice that $\phi_\sigma^*(s) = \phi_0^*(s-\sigma)$. Finally, $\sigma_0$ will be determined in such a way that there exists a unique solution $\xi_{\sigma_0}\in H^2_\text{per}(0,2\pi)$ of 
$$
L_{u_0}\xi_{\sigma_0} =-\iu e(\cdot+\sigma_0)
$$
with the property that $\xi_{\sigma_0}\perp_{L^2} u_0'$. Details of the construction of $\sigma_0$ and $\xi_{\sigma_0}$ will be given in Lemma~\ref{one_d_kernel}.

\begin{Theorem} \label{Fortsetzung_nichttrivial} Let $d \in \R\setminus\{0\}$, $\zeta,\omega,f_0\in \R$ and $e \in H^2(0,2\pi)$ be fixed. Let furthermore $u_0\in H^2_{\text{per}}(0,2\pi)$ be a non-trivial non-degenerate solution of \eqref{TWE} for $f_1=0$. If $\sigma_0\in \R$ satisfies
\begin{equation} \label{eq:sigma_0}
\IT \int_0^{2\pi}e(s+\sigma_0)\overline{\phi_0^\ast(s)}\,ds =0
\end{equation}
and
\begin{equation} \label{eq:transv}
\IT \int_0^{2\pi}e'(s+\sigma_0)\overline{\phi_0^\ast(s)}\,ds \neq 0
\end{equation}
then the maximal continuum $\mathcal{C}^+\subset [0,\infty) \times H^2_{\text{per}}(0,2\pi)$ of solutions $(f_1,u)$ of \eqref{TWE} with $(0,u_0)\in\mathcal{C}^+$ has the following properties:
\begin{itemize}
\item[(i)] there exists a smooth curve $C:[0,\delta) \to \mathcal{C}^+$ with $C(t)=(f_1(t), u(t))$, $\dot{f_1}(0)=1$, $C(0)=(0,u_{\sigma_0})$ such that locally near $(0,u_{\sigma_0})$ all solutions $(f_1,u)$ of \eqref{TWE} with $f_1\geq 0$ lie on the curve $S$ or on the curve $C$,
\item[(ii)] $\mathcal{C}^+\cap [0,M]\times H^2_{\text{per}}(0,2\pi)$ is bounded for any $M>0$.
\end{itemize} 
Moreover, if zero is an algebraically simple eigenvalue of $L_{u_0}$ and if furthermore
\begin{equation} \label{further_cond}
\begin{split} 
2\RT \int_0^{2\pi} \bigl(2u_0|&\xi_{\sigma_0}|^2+\overline{u_0} \xi_{\sigma_0}^2\bigr)\overline{\phi_0^*} \, ds \RT \int_0^{2\pi} \bigl(u_0' \overline{u_0}+2 u_0\overline{u_0'}\bigr) u_0' \overline{\phi_0^*} \, ds \\
&\neq  \biggl(\IT \int_0^{2\pi} e'(s+\sigma_0) \overline{\phi_0^*(s)} \, ds \biggr)^2,
\end{split}
\end{equation}
then there exists a connected set $\mathcal{C}^{+}_* \subset \mathcal{C}^+$ with $\operatorname{pr}_1(\mathcal{C}^{+}_*)\subset (0,\infty)$ and $(0,u_{\sigma_0})\in \overline{\mathcal{C}^{+}_*}$ which satisfies at least one of the following properties:
\begin{enumerate}[(a)]
\item $\operatorname{pr}_1(\mathcal{C}^+_*)=(0,\infty)$,
\end{enumerate}
or 
\begin{enumerate}[(b)]
\item $\exists u_0^+ \not = u_{\sigma_0}: \, (0, u_0^+) \in \overline{\mathcal{C}^{+}_*}$.
\end{enumerate}
A maximal continuum  $\mathcal{C}^-\subset (-\infty,0]\times H^2_{\text{per}}(0,2\pi)$ with corresponding properties also exists.
\end{Theorem}

\medskip

For the special choice $e(s)=\eu^{\iu k_1 s}$ Theorem \ref{Fortsetzung_nichttrivial} takes the following form. 

\begin{Corollary} \label{Fortsetzung_nichttrivial_Korollar} Let $k_1 \in \N$, $e(s)=\eu^{\iu k_1 s}$ and $d,\zeta,\omega,f_0,u_0$ be as in Theorem~\ref{Fortsetzung_nichttrivial}. Assume that
\begin{equation} \label{extra}
\int_0^{2\pi} \eu^{\iu k_1 s}\overline{\phi_0^\ast(s)}\,ds \neq 0
\end{equation}
and that $\sigma_0\in\R$ satisfies
\begin{equation} \label{eq:sigma_0_Korollar}
\tan(k_1\sigma_0) = \frac{\int_0^{2\pi}\cos(k_1s)\IT \phi_0^*(s)-\sin(k_1 s)\RT \phi_0^*(s)\,ds}{\int_0^{2\pi}\sin(k_1s)\IT \phi_0^*(s)+\cos(k_1 s)\RT \phi_0^*(s)\,ds}.
\end{equation}
Then the conditions \eqref{eq:sigma_0} and \eqref{eq:transv} of Theorem~\ref{Fortsetzung_nichttrivial} hold.
\end{Corollary}

\begin{Remark} \label{Fortsetzung_nichttrivial_Remark}
($\alpha$) It follows from the implicit function theorem that in the setting of Theorem~\ref{Fortsetzung_nichttrivial} Assumption \eqref{eq:sigma_0} is a necessary condition for bifurcation (non-trivial kernel of the linearization). Assumption \eqref{eq:transv} amounts to the transversality condition. In the setting of Corollary \ref{Fortsetzung_nichttrivial_Korollar} this means that, if \eqref{extra} is satisfied, assumption \eqref{eq:sigma_0_Korollar} is a necessary condition for bifurcation. \\
($\beta$) Assumption \eqref{extra} in Corollary \ref{Fortsetzung_nichttrivial_Korollar} guarantees that the numerator and the denominator of the right-hand side of \eqref{eq:sigma_0_Korollar} do not vanish simultaneously. In the case where the denominator vanishes, Equation \eqref{eq:sigma_0_Korollar} is to be read as $\cos(k_1\sigma_0)=0$. 
In the interval $[0,\tfrac{\pi}{k_1})$ equation \eqref{eq:sigma_0_Korollar} has a unique solution $\sigma_0\in [0,\tfrac{\pi}{k_1})$. All solutions of \eqref{eq:sigma_0_Korollar} in $[0,2\pi)$ are then given by 
$\sigma_0+j\tfrac{\pi}{k_1}$ for $j=0,\ldots,2k_1-1$. This can result in up to $2 k_1$ bifurcation points. Smaller periodicities of $u_0$ may reduce the actual number of different bifurcation points. E.g., if $k_1\geq 2$ and if $u_0$ has smallest period $\tfrac{2\pi}{k_1}$ then only two bifurcation points exist. \\
($\gamma$) Let $j\in \N$ not be a divisor of $k_1$ and $u_0$ be $\tfrac{2\pi}{j}$-periodic. Then assumption \eqref{extra} is not satisfied since $\phi_0^*$ inherits the periodicity of $u_0$. We will say more about this case in the Appendix. \\
($\delta$) The non-trivial solutions $u_0$ of \eqref{TWE} for $f_1=0$ and $\omega=0$ constructed in \cite{Gaertner_et_al},\cite{Mandel} are even around $s=0$. In this case, \eqref{further_cond} is not an additional assumption because it coincides with assumption \eqref{eq:transv}. The reason is that
$\phi_0^*$ (spanning $\ker L_{u_0}^*$) inherits the parity of $u_0'$ (spanning $\ker L_{u_0}$) which implies $\int_0^{2\pi} \bigl(u_0' \overline{u_0}+2 u_0\overline{u_0'}\bigr) u_0' \overline{\phi_0^*} \, ds=0$, 
cf. Proposition \ref{parity}. Also, the value of $\sigma_0$ in Corollary~\ref{Fortsetzung_nichttrivial_Korollar} is determined by the simpler expression 
$$
\tan(k_1\sigma_0) = -\frac{\int_0^{2\pi} \sin(k_1 s)\RT \phi_0^*(s)\,ds}{\int_0^{2\pi}\sin(k_1s)\IT \phi_0^*(s)\,ds}.
$$
It is an open problem if \eqref{TWE} admits solutions for $f_1=0$ and $\omega=0$ which (up to a shift) are not even around $s=0$.\\
($\epsilon$) Note that in property (b) we exclude that $u_0^+=u_{\sigma_0}$ but we do not exclude that $u_0^+$ coincides with a shift of $u_0$ different from $u_{\sigma_0}$. 
\end{Remark}

\subsection{Two-sided continuations} \label{sec:two_sided}

Here we explain how we can use the results of Theorem~\ref{Hauptresultat} and Theorem~\ref{Fortsetzung_nichttrivial}, Corollary~\ref{Fortsetzung_nichttrivial_Korollar} for the continua $\mathcal{C}^+$ and $\mathcal{C}^-$ in order to obtain two-sided continua w.r.t. the parameter component $f_1$. 

\medskip

As a first trivial observation we can construct a two-sided continuum in the following way both for the setting of Theorem~\ref{Hauptresultat} and Theorem~\ref{Fortsetzung_nichttrivial}: let $\mathcal{C}\subset \R\times H^2_{\text{per}}(0,2\pi)$ be the maximal continuum of solutions $(f_1,u)$ of \eqref{TWE} with $(0,u_0)\in \mathcal{C}$. Then $\mathcal{C}$ contains both $\mathcal{C}^+$ and $\mathcal{C^-}$.

\medskip

Next we assume that the generalized forcing term $f(s)=f_0+f_1 e(s)$ satisfies the symmetry condition that $e\bigl(s+\frac{\pi}{k_1}\bigr)=-e(s)$ for some $k_1\in \N$. This symmetry condition is motivated by \eqref{TME} where $e(s)=\mathrm{e}^{\mathrm{i}k_1 s}$. If we denote by $R$ the reflection operator which acts on solution pairs and is given by 
$$
R: (f_1,u) \mapsto \bigl(-f_1, u\bigl(\cdot+\tfrac{\pi}{k_1}\bigr)\bigr)
$$
then, again both for the setting of Theorem~\ref{Hauptresultat} and Theorem~\ref{Fortsetzung_nichttrivial}, the continuum $\mathcal{C}$ has the following property:
\begin{equation*}
(f_1,u)\in \mathcal{C} \Leftrightarrow  R(f_1,u)\in\mathcal{C}.  
\end{equation*}
This shows that globally the solution sets for positive and negative $f_1$ only differ by a phase shift. The following global structure result is a consequence of this symmetry. 

\begin{Proposition} \label{maximal_continua}
Let $d \in \R\setminus\{0\}$, $\zeta,\omega,f_0\in\R$ and $e\in H^2(0,2\pi)$ be such that $e\bigl(s+\frac{\pi}{k_1}\bigr)=-e(s)$ for some $k_1 \in \N$. Let furthermore $u_0$ be a solution of \eqref{TWE} for $f_1=0$. Then the maximal continua $\mathcal{C}^+$, $\mathcal{C}^-$ and $\mathcal{C}$ containing $(0,u_0)$ satisfy
$
\mathcal{C^-}=R(\mathcal{C}^+)
$
and
$
\mathcal{C}\supset\mathcal{C}^+ \cup \mathcal{C}^-.
$ 
\end{Proposition}
\begin{proof}
It is obvious that $\mathcal{C}\supset\mathcal{C}^+ \cup \mathcal{C}^-$. Now we prove that $\mathcal{C^-}=R(\mathcal{C}^+)$. Clearly, $\mathcal{C}^+$ and $R(\mathcal{C}^+)$ contain all shifts $\{(0,u_{\sigma}):\sigma\in\R\}$. Since additionally $R(\mathcal{C}^+) \subset (-\infty,0]\times H^2_{\text{per}}(0,2\pi)$ is connected we find that $R(\mathcal{C}^+) \subset \mathcal{C}^-$. If we assume that $R(\mathcal{C}^+) \subsetneq \mathcal{C}^-$ then we obtain $\mathcal{C}^+ \subsetneq R^{-1}(\mathcal{C}^-)$, which contradicts the maximality of $\mathcal{C}^+$.  
\end{proof}

\noindent As another consequence, we have that either $\operatorname{pr}_1(\mathcal{C})=(-\infty,\infty)$ or $\operatorname{pr}_1(\mathcal{C})$ is bounded from above and below. In the latter case, we call $\mathcal{C}$ a loop.

\medskip

Our final result builds upon Theorem~\ref{Hauptresultat} and the resulting two-sided continuation of a trivial solution $u_0$. It describes the shape of the $L^2$-projection of the continuum $\mathcal{C}$ locally near $(0,u_0)$. In particular, local convexity or concavity can be read from this result. In Section~\ref{sec:numerical} we will put this result into perspective with numerical simulations of the $f_1$-continuation of trivial solutions.

\begin{Theorem} \label{second_derivative} Assume that the assumptions of Theorem~\ref{Hauptresultat} are satisfied and that additionally $e(s)=\eu^{\iu k_1 s}$ is fixed for a $k_1 \in\N$. Then we can determine the local shape of the curve $f_1\mapsto \|u(f_1)\|_2^2$ as follows:
$$
\frac{d}{df_1} \|u(f_1)\|_2^2\mid_{f_1=0} = 0, \qquad \frac{d^2}{df_1^2} \|u(f_1)\|_2^2\mid_{f_1=0}= 4\pi(\RT(u_0 \overline \epsilon)+|\alpha|^2+|\beta|^2)
$$
with 
\begin{align*}
\alpha &= \frac{-\mathrm{i}(dk_1^2+k_1\omega+\zeta+\mathrm{i}-2|u_0|^2)}{(\zeta+dk_1^2-2|u_0|^2)^2-(\omega k_1+\mathrm{i})^2-|u_0|^4},\\
\beta &= \frac{\mathrm{i} u_0^2}{(\zeta+dk_1^2-2|u_0|^2)^2-(\omega k_1-\mathrm{i})^2-|u_0|^4},\\
x&=\zeta-\mathrm{i}-2|u_0|^2,\\
y& =-u_0^2,\\
z& = 4u_0(|\alpha|^2+|\beta|^2)+4\overline{u_0}\alpha\beta,\\
\epsilon &= \frac{-\overline z y+z\overline x}{|x|^2-|y|^2}.
\end{align*}
\end{Theorem}

\section{Numerical Illustration of the Analytical Results} \label{sec:numerical}

In this section we restrict ourselves to equation \eqref{TME}, i.e., we fix $e(s)=\eu^{\iu k_1 s}$.
For this choice, we know from Section~\ref{sec:two_sided} that the one-sided continua
$\mathcal{C}^+$ and $\mathcal{C}^-$ are related by $\mathcal{C}^-= R(\mathcal{C}^+)$. The following numerical examples were computed with $ d=-0.1 $, $f_0=2$, $k_1=1$, and $\omega=1$.

\begin{figure}[h]
 \includegraphics[width=0.9\textwidth]{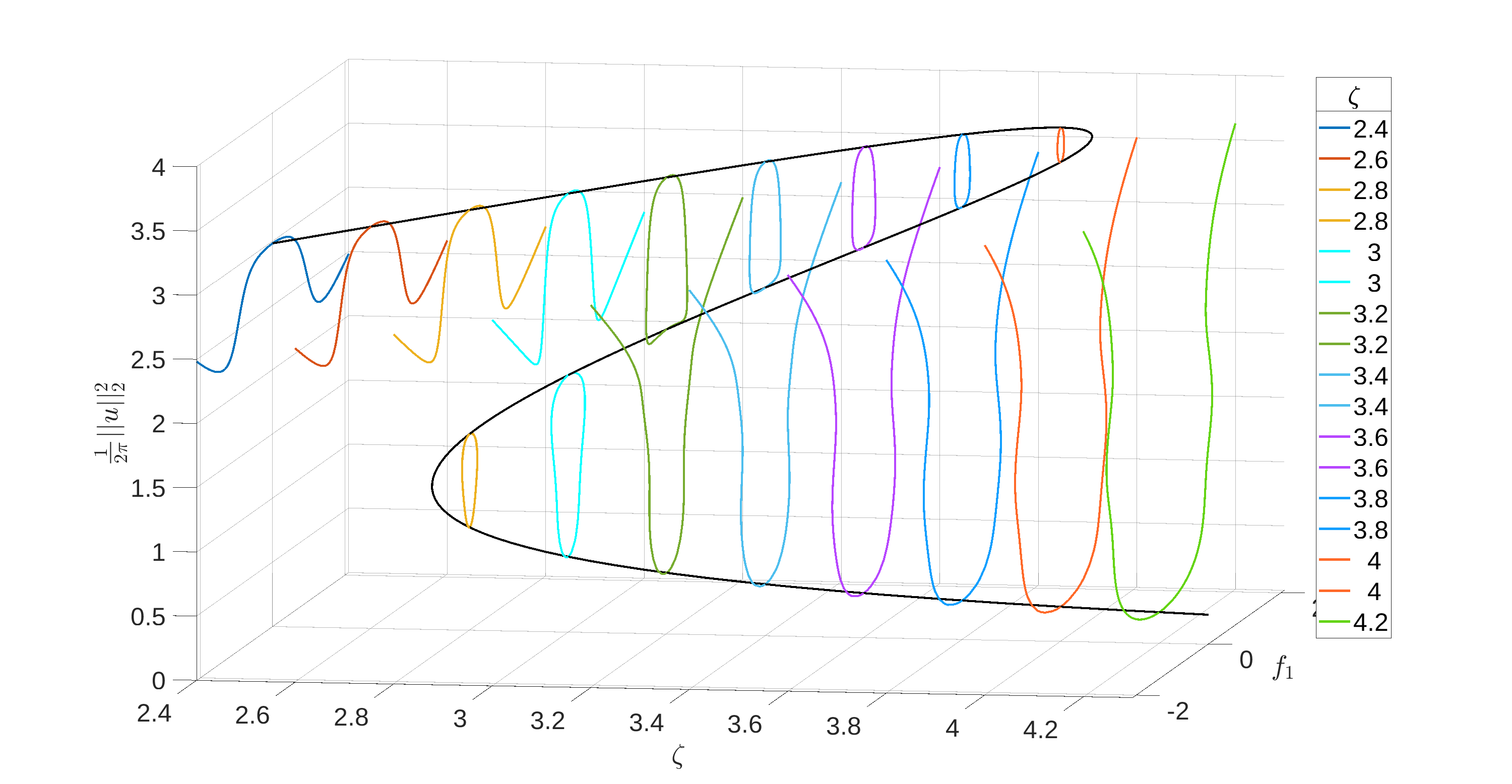} 
 \caption{
 Continua of solutions $(f_1,u)$ of \eqref{TME} for selected values of the detuning $\zeta$.
 The other parameters were set to $ d=-0.1 $, $f_0=2$, $k_1=1$, and $\omega=1$.} 
  \label{Continua.few.lines}
\end{figure}

Figure~\ref{Continua.few.lines} illustrates some of the two-sided continua $\mathcal{C}^+ \cup \mathcal{C}^-$ 
obtained by continuation of trivial solutions for different values of the detuning $\zeta$.
Every point on the black and colored curves corresponds to a solution $u$ of \eqref{TME}, but for the sake of visualization in a three-dimensional image every solution has to be represented by a single number. 
In Figure~\ref{Continua.few.lines}, the quantity $\tfrac{1}{2\pi}\|u\|_2^2$ was used for this purpose. 

The black curve corresponds to spatially constant solutions of \eqref{TME} obtained for $ f_1=0$
and $ \zeta\in[2.4,4.3]$.
The colored curves represent (parts of) the continua associated to these solutions. Every trivial solution (possibly except the ones at turning points) has an associated continuum, but for the sake of visualization these continua are only shown for selected values of $\zeta$, namely $\zeta\in\{2.4, 2.6, \ldots, 4.0, 4.2\}$.
The picture is symmetric with symmetry plane $\{(\zeta,0,z) \colon \zeta\in\R, z\in\R\}$.
This is an immediate consequence of the relation $\mathcal{C}^-= R(\mathcal{C}^+)$ and the fact that shifting $u$ does not change $\|u\|_2.$

For $ \zeta \in\{2.4, \, 2.6, \, 4.2\} $ there is only one trivial solution, and for these three values Figure~\ref{Continua.few.lines} shows a part of the associated two-sided continuum $\mathcal{C}^+ \cup \mathcal{C}^-$. 
Although $f_1$ was restricted to $[-2,2]$, each of these continua appears to be global in $f_1$, i.e.~we conjecture that the continua continue for \emph{all} values $f_1\in (-\infty,\infty)$.
This corresponds to case (a) in Theorem~\ref{Hauptresultat}.

For $ \zeta \in\{2.8, \, 3.0, \ldots, \, 4.0\} $, however, there are three trivial solutions. For these values of $\zeta$, there is one colored loop which connects two solutions, and one continuum which seems to continue for all values of $f_1$. 
The former corresponds to case (b) in Theorem~\ref{Hauptresultat}, the latter to case (a).
For $ \zeta \in\{2.8, \, 3.0\} $ the ``lower'' two solutions are connected, whereas for 
$ \zeta \in\{3.2, \, \ldots, \, 4.0\} $ it is the ``upper'' two solutions which are connected. 
Hence, there seems to be a threshold value $\zeta^*$ that determines which of the two scenarios occurs.
Computations with more values of $\zeta$ show that this threshold value $\zeta^*$ lies between 3.1344 and 3.1359;
cf.~Figure~\ref{Switching.point}. The union of the continua for $\zeta$-values close to the threshold $\zeta^*$ (i.e.\ for $ \zeta=3.1344 $ and $ \zeta=3.1359 $) is nearly the same, and the two continua nearly meet in two points.\footnote{As mentioned earlier, only the $L^2$-norm of solutions can be visualized in 
Figure~\ref{Continua.few.lines}, \ref{Switching.point} and all other plots. The fact that two functions have (nearly) the same norm does, of course, not imply that the functions themselves are (nearly) identical. 
It can be checked, however, that the two solutions which correspond to the two points where the distance between the two continua is minimal are indeed very similar (data not shown).}
The mathematical mechanisms which cause this qualitative change are not yet understood. 
One could expect that the connectivity threshold coincides with the value where the square of the $L^2$-norm of the solutions as a function of $f_1$ changes from being locally convex to locally concave.
However, Theorem~\ref{second_derivative} shows that this is \emph{not} true.

\begin{figure}[h]
 \includegraphics[width=0.9\textwidth]{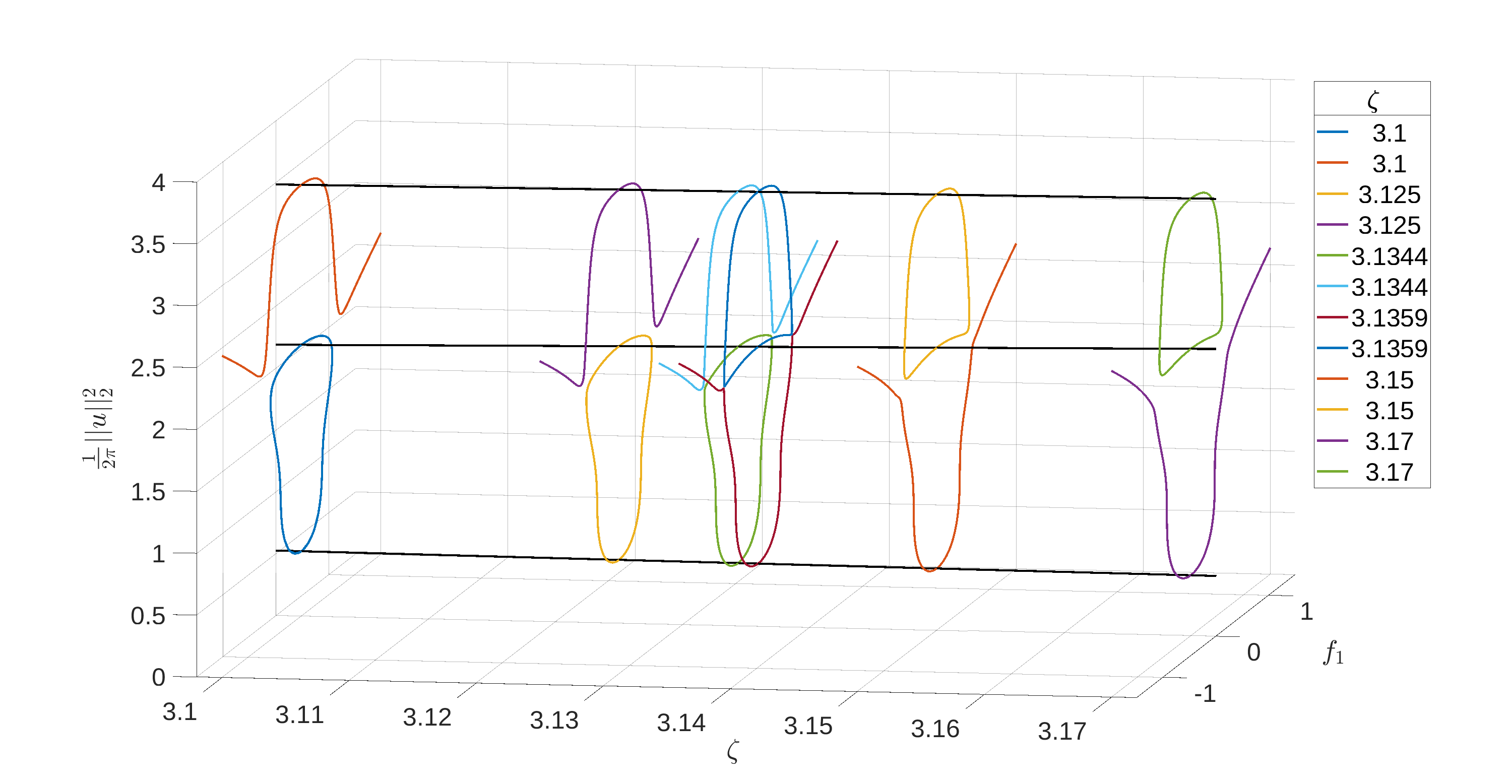} 
 \caption{
 Same situation as in Figure~\ref{Continua.few.lines}. Zoom to the region close to the threshold where the continua change connectivity.} 
  \label{Switching.point}
\end{figure}

Figure~\ref{Fig02} illustrates the same application, but depicted from a different angle and with more values of $\zeta$.
Repeating the simulation with $d=0.1$ (anomalous dispersion) instead of $d=-0.1$ (normal dispersion) did not change the picture essentially.

\begin{figure}[h]
\includegraphics[width=0.9\textwidth]{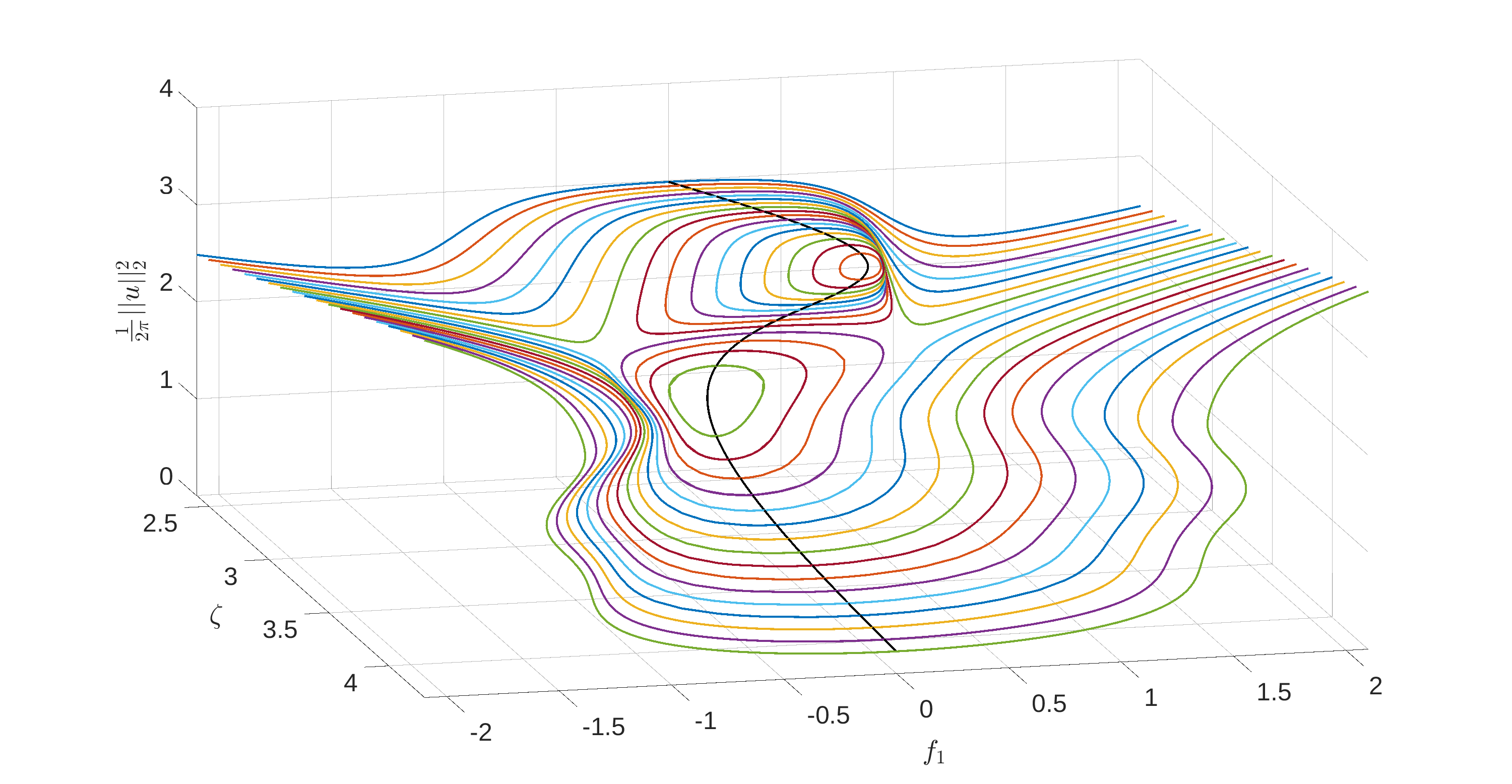} 
 \caption{Same situation as in Figure~\ref{Continua.few.lines}, but depicted from a different angle and with more values of $\zeta$.
 }
  \label{Fig02}
\end{figure}

Figures~\ref{Continua.few.lines}, \ref{Switching.point}, and \ref{Fig02} were generated by 
discretizing \eqref{TME} with central finite differences ($1000$ grid points),
and by applying the classical continuation method as described in, e.g., \cite{allgower-georg-1990}, 
to the discretized system.

The result of Theorem~\ref{second_derivative} can be interpreted as follows: each point on the trivial curve is a local extremum of the squared $L^2$-norm of the solution curve $f_1\mapsto u(f_1)$. The type of local extremum is described by the sign of the second derivative $\frac{d^2}{df_1^2} \|u(f_1)\|_2^2\mid_{f_1=0}$. We visualize this by an example for $d=-0.1$, $f_0=2$, $k_1=1$, $\omega=1$. By using the parameterization $t\mapsto \zeta(t), t\mapsto u_0(t)$ for $t\in (-1,1)$ from \eqref{parametrization} we can illustrate the sign-changes of the second derivative. In Figure~\ref{Fig_dummy} we are plotting the curve $t \mapsto (\zeta(t), |u_0(t)|^2)$ and indicate at each point on the curve the sign of $4\pi(\RT(u_0(t) \bar \epsilon(t))+|\alpha(t)|^2+|\beta(t)|^2)$, where $\epsilon(t), \alpha(t), \beta(t)$ are taken from Theorem~\ref{second_derivative} with $\zeta=\zeta(t)$ and $u_0=u_0(t)$. In this particular example, as we run through the curve of trivial solutions from left to right a first sign-change of $\frac{d^2}{df_1^2} \|u(f_1)\|_2^2\mid_{f_1=0}$ occurs at $\zeta\approx 0.8533$.

\begin{wrapfigure}[14]{r}{0.55\textwidth}
\centering
\includegraphics[width=0.5\textwidth]{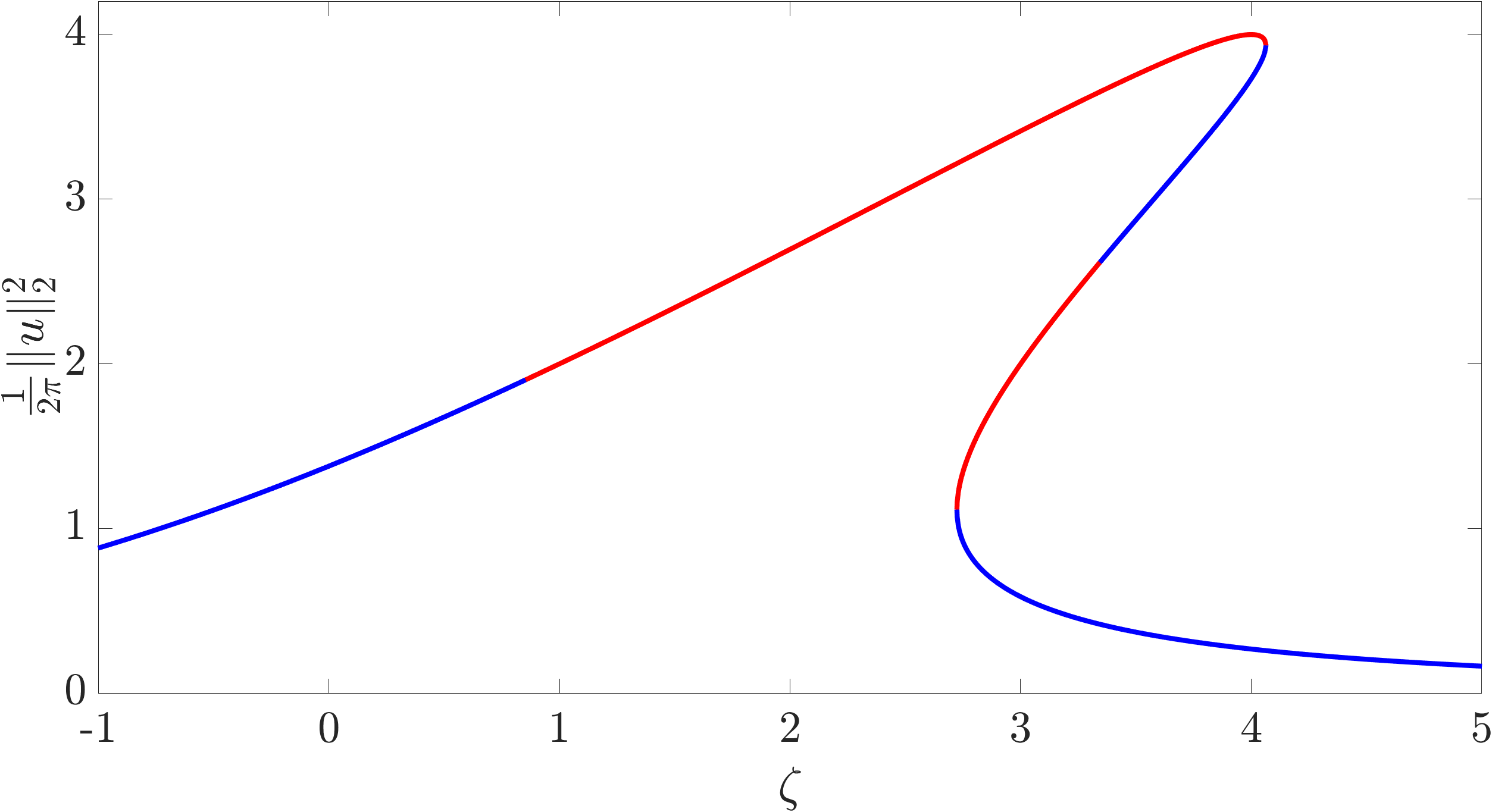} 
 \caption{Sign of the second derivative of $f_1\mapsto \|u(f_1)\|_2^2$ at $f_1=0$; blue$=$positive, red$=$negative.}
  \label{Fig_dummy}
\end{wrapfigure}
 A second sign-change (in fact a singularity changing from $-\infty$ to $+\infty)$ occurs at the first turning point. Then, the next sign-change occurs on the part of the branch between the two turning points at $\zeta\approx 3.34$. Finally, the second turning point generates the last sign-change from $-\infty$ to $+\infty$. Clearly, the changes in the nature of the local extremum of $f_1\mapsto \|u(f_1)\|_2^2$ at $f_1=0$ do not correspond to the topology changes of the solution continua which occur near the threshold value $\zeta^*\in (3.1344,3.1359)$. 

%

Next, we keep the parameters $d=-0.1 $, $f_0=2$, $k_1=1$ but choose $\omega=0$ instead of $\omega=1$. Recall that for $\omega=0$ there is a plethora of non-trivial solutions of \eqref{TME} for $f_1=0$, cf.~\cite{Gaertner_et_al},\cite{Mandel}. In fact, this time we find additional primary and secondary bifurcation branches for $f_1=0$ which are illustrated in Figure~\ref{Continua_omega_zero} in grey and brown, respectively. Bifurcation points are shown as grey dots. The bifurcation branches consist of non-trivial solutions. Further, some numerical approximations of the two-sided maximal continua $\mathcal{C}$ obtained by continuation
of trivial or non-trivial solutions for different values of the detuning $\zeta$ are shown. If we start from a constant solution at $f_1=0$, then $\mathcal{C}^\pm$ are described by Theorem~\ref{Hauptresultat}. Likewise, if we start from a non-constant solution at $f_1=0$ which has no smaller period than $2\pi$, then $\mathcal{C}^\pm$ are described by Theorem~\ref{Fortsetzung_nichttrivial}. In both cases, $\mathcal{C}\supset\mathcal{C}^+ \cup \mathcal{C}^-$ by Proposition~\ref{maximal_continua}, but in all examples below we observe in fact equality. If we expect a maximal continuum to contain two or more (non-trivial) different simple closed curves, then we illustrate the latter ones with different colors. Let us look at some particular values of $\zeta$ where different phenomena occur.

\begin{figure}[ht!]
\centering
\includegraphics[width=0.85\textwidth]{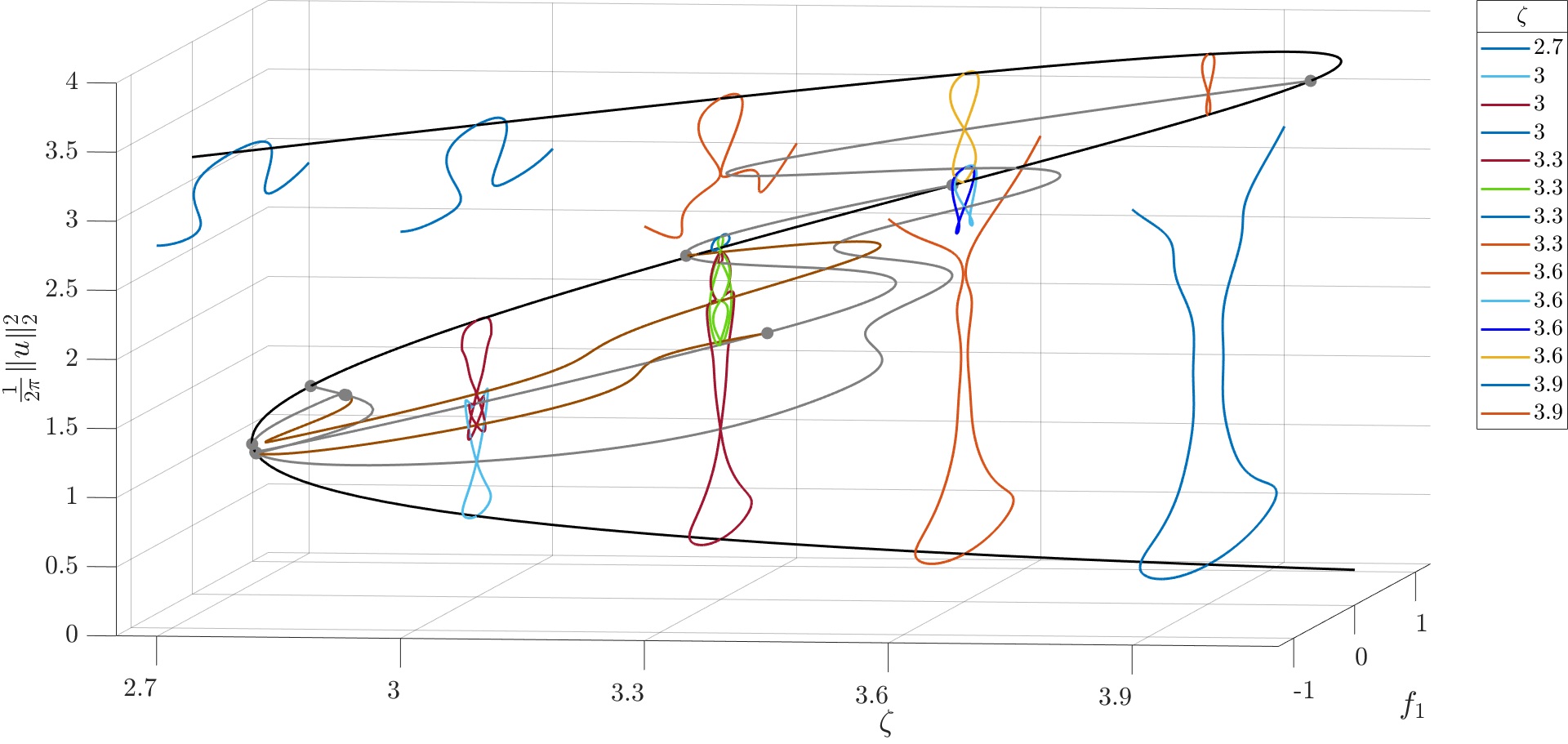} 
 \caption{Continua of solutions $(f_1,u)$ of \eqref{TME} for selected values of the detuning $\zeta$.
 The other parameters were set to $d=-0.1 $, $f_0=2$, $k_1=1$, and $\omega=0$.}
  \label{Continua_omega_zero}
\end{figure}

At $\zeta=2.7$ we see exactly one solution for $f_1=0$. This solution is constant and its continuation appears to be global in $f_1$. For $\zeta=3.9$ and $f_1=0$ we see three constant solutions but also one non-constant solution (up to shifts) which lies on one of the grey bifurcation branches. The continuation of the constant solution with smallest magnitude again appears to be global in $f_1$, while the other three solutions lie on the same eight-shaped maximal continuum which we will denote as \textit{figure eight continuum}. Note that the latter continuum contains all shifts of the non-trivial solution for $f_1=0$. 

The figure eight can be interpreted as an outcome of Theorem~\ref{Hauptresultat} applied to one of the constant solutions on the figure eight. Here, case (b) of the theorem applies. However, the figure eight can also be interpreted as an outcome of Theorem~\ref{Fortsetzung_nichttrivial} applied to the non-constant solution $u_0$ at $f_1=0$. Again, case (b) of the theorem applies. A plot (which we omit) of the non-trivial solution $u_0$ at $f_1=0$ shows that $u_0$ has no smaller period than $2 \pi$. Thus, according to Remark~\ref{Fortsetzung_nichttrivial_Remark}.($\beta$) exactly two shifts of it, which differ by $\pi$, are bifurcation points. To sum up, we observe that the figure eight continuum in fact contains a simple closed figure eight curve which exactly goes through two shifts of $u_0$ (which differ by $\pi$) in the point where the orange lines intersect the grey line of non-trivial solutions. The two shifts cannot be distinguished in the picture, because a shift does not change the $L^2$-norm. 
\begin{figure}[h]
\centering
\includegraphics[width=0.8\textwidth]{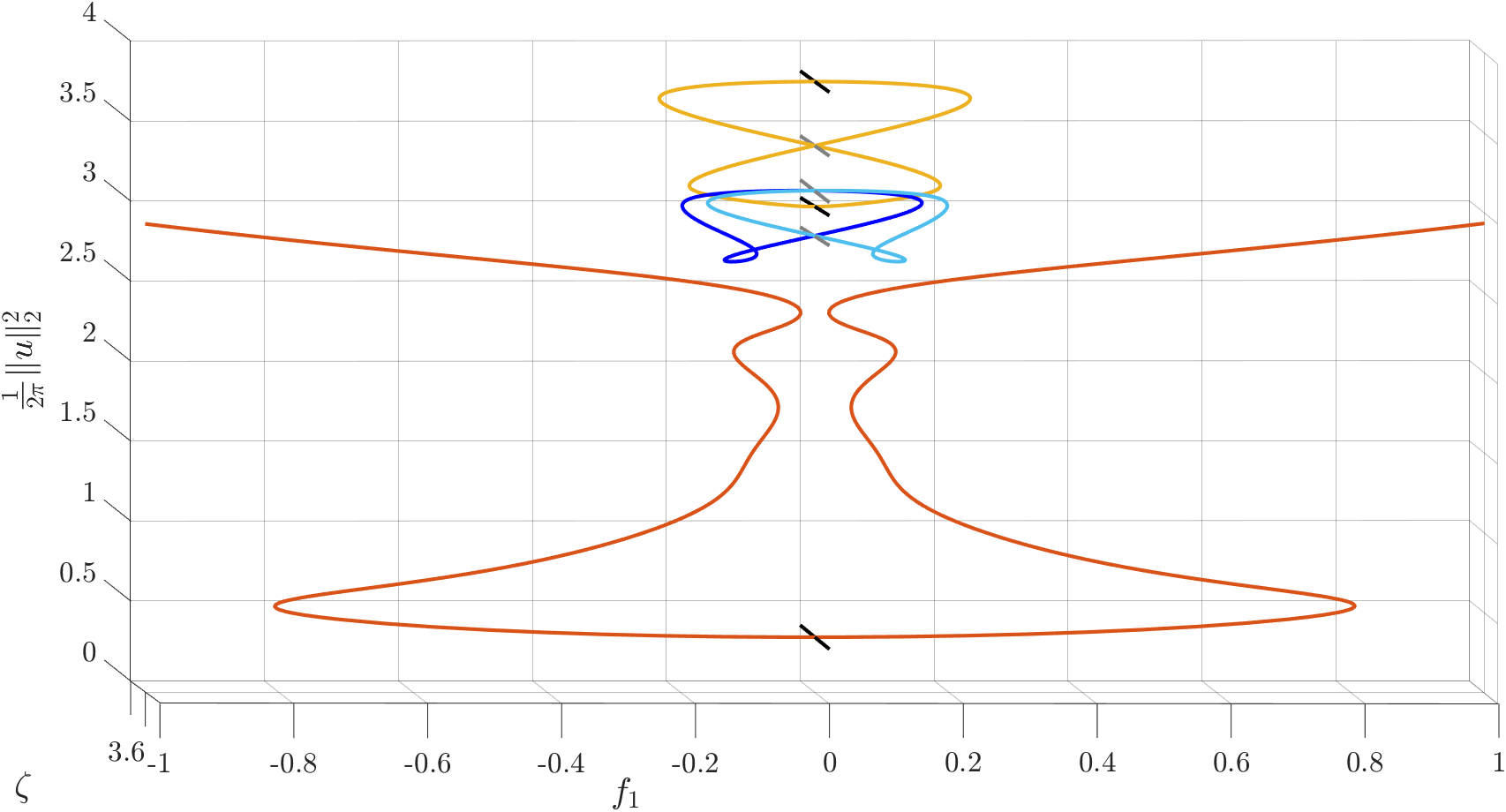} 
 \caption{Zoom at $\zeta=3.6$.}
  \label{Zoom_bei_drei_sechs}
\end{figure}
To illustrate the different continua for $\zeta=3.6$, we provide a zoom in Figure~\ref{Zoom_bei_drei_sechs}. We obtain again an unbounded continuum and a figure eight continuum. However, here we also find a third maximal continuum which cannot be found by simply continuing one of the constant solutions. This continuum consists of the blue and the light blue simple closed curve connected to each other by shifts at $f_1=0$. The parts of the blue and the light blue curve in the region $f_1\geq 0$ are described by case (b) of Theorem~\ref{Fortsetzung_nichttrivial} applied to one of the non-trivial solutions $u_0$ at $f_1=0$ on it. They have no smaller period than $2\pi$ (plots not shown). Going from the blue part to the light blue part is a consequence of reflection. At $f_1=0$ the blue curve intersects the grey line at exactly two points. The light blue curve does the same, but at $\pi$-shifts of these points.

For $\zeta=3.3$ the situation is more complicated. In this case, we see three constant solutions for $f_1=0$ but also seven non-constant ones. The continuation of the upper constant solution (orange) appears to be unbounded. We observe that the blue, the red and the green simple closed curve in fact form a single maximal continuum, since all curves are connected by shifts of non-constant solutions at $f_1=0$. Viewed from top to bottom, we find (plots not shown) that the first, the third and the last one are $\pi$-periodic while the remaining ones have smallest period $2\pi$. All together, we observe that exactly two shifts of every non-constant solution at $f_1=0$ are bifurcation points. For the solutions which have no smaller period than $2\pi$ this is a direct consequence of Theorem~\ref{Fortsetzung_nichttrivial}, cf.~Remark~\ref{Fortsetzung_nichttrivial_Remark}.($\beta$). However, at the three remaining $\pi$-periodic solutions at $f_1=0$ Theorem~\ref{Fortsetzung_nichttrivial} does not apply, cf.~Remark~\ref{Fortsetzung_nichttrivial_Remark}.($\gamma$). Nevertheless, we observe continuations from these points. Interestingly, these points seem to be characterized by horizontal tangents, at least in this example.

\begin{figure}[ht!]
\centering
\includegraphics[width=0.8\textwidth]{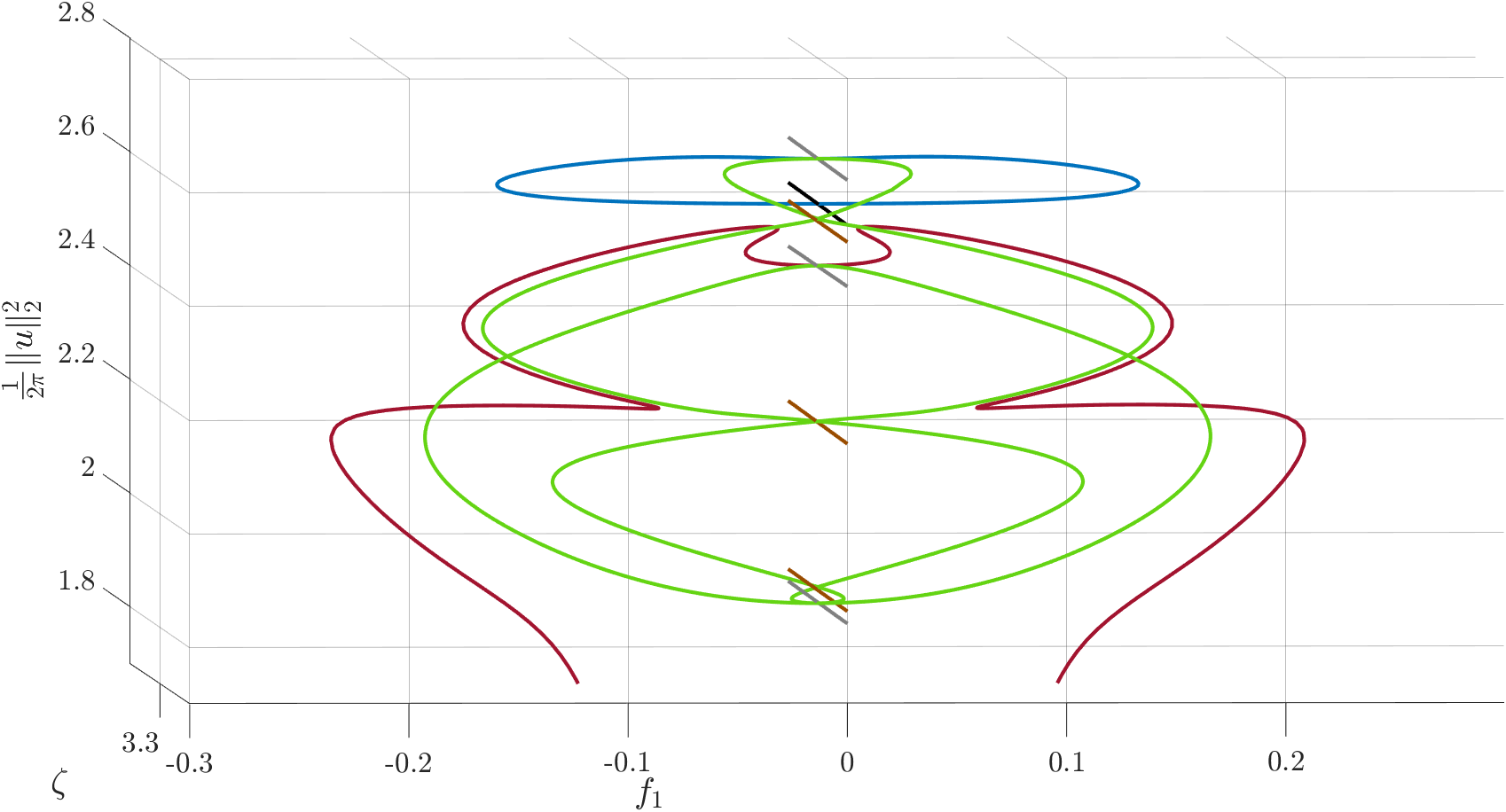} 
 \caption{Zoom at $\zeta=3.3$.}
  \label{Zoom_bei_drei_drei}
\end{figure}

\begin{figure}[ht!]
\centering
 \includegraphics[width=\textwidth]{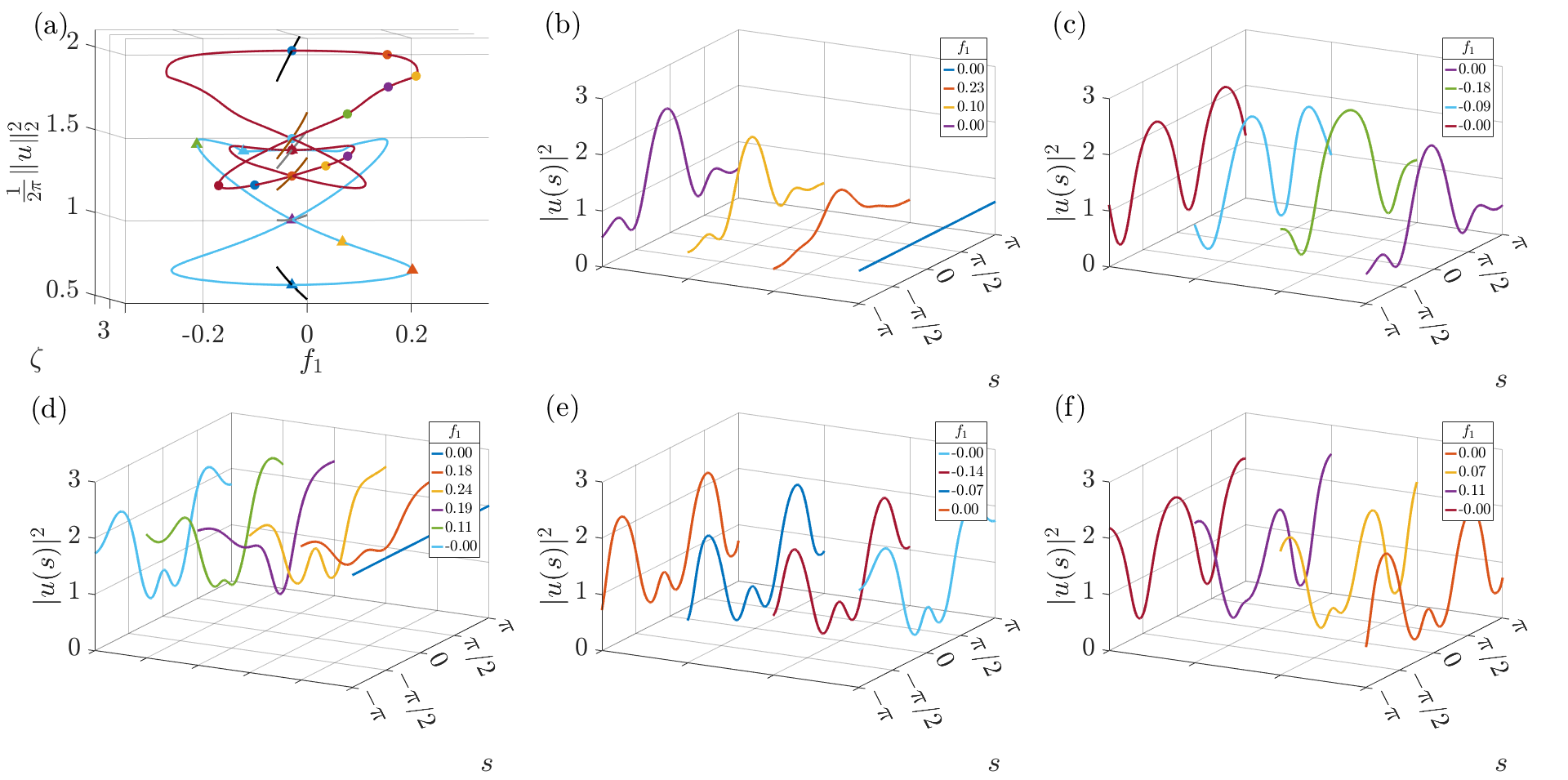} 
 \caption{Zoom at $\zeta=3$ and illustration of selected functions.} 
  \label{Zoom_bei_drei}
\end{figure}

For $\zeta=3$ we see three constant solutions and four non-constant ones at $f_1=0$. Again, the continuation of the upper constant solution is unbounded.  We provide a more general investigation in Figure~\ref{Zoom_bei_drei}, where we also depict several of the continued solutions $u$ of \eqref{TME} for $f_1\neq 0$. Since $u$ is complex-valued, we use the quantity $|u(s)|^2$ for illustration purposes and plot it against $s\in[-\pi,\pi]$. In Figure~\ref{Zoom_bei_drei}(a) we show a bounded continuum consisting of the light blue and the red simple closed curve connected to each other by shifts at $f_1=0$. Starting from the constant solution on the light blue curve and proceeding first into the $f_1>0$ direction, Figure~\ref{Zoom_bei_drei}(b)-(c) show plots of functions corresponding to colored triangles. In Figure~\ref{Zoom_bei_drei}(d)-(f) functions corresponding to colored dots on the red curve are shown, where we start again at the constant solution and initially proceed in the $f_1>0$ direction. We observe that both curves cross the ($\pi$-periodic) non-constant solution with second largest norm, but at two different shifts: the leftmost dark-red curves in (c) and (f) only coincide after a non-zero shift. Continuations from $\pi$-periodic solutions at $f_1=0$ are not covered by Theorem~\ref{Fortsetzung_nichttrivial}. Nevertheless, they are observed in the numerical experiments, again with horizontal tangents. The explanation of these continuations remains open, cf.~the Appendix for further discussion.




\section{Proof of a-priori bounds} \label{sec:a-priori} 

We use the notation $r_+=\max\{0,r\}$ to denote the positive part of any real number $r\in \R$ and also $\mathbf{1}_{d<0}$ to denote (as a function of $d\in \R$) the characteristic function of the interval $(-\infty,0)$. We write $\|\cdot\|_p$ for the standard norm on $L^p(0,2\pi)$ for $p \in[1,\infty]$. A continuous map between two Banach spaces is said to be compact if it maps bounded sets into relatively compact sets. 

\begin{Theorem}\label{a-priori}
Let $d \in \R\setminus\{0\}$, $\zeta,\omega\in\R$ and $f \in H^2(0,2\pi)$. Then for every solution $u \in H^2_{\text{per}}(0,2\pi)$ of \eqref{TWE} the a-priori bounds
\begin{align}
\|u\|_2 & \leq F, \label{l2}\\
\|u'\|_2 & \leq B\|u\|_2^{\frac{1}{4}}\leq B F^{\frac{1}{4}} , \label{l2_prime} \\
\|u\|_\infty & \leq C \label{linfinity}
\end{align}
hold, where 
\begin{align*}
F&= F(f)= \|f\|_2, \\  
B&= B(d,f)=\frac{F^{\frac{11}{4}}}{2|d|}+2\|f'\|_{\infty} F^{\frac{1}{4}}+\sqrt{\|f''\|_2 F^{\frac{1}{2}}+2\|f'\|_{\infty}\bigg(\sqrt{\frac{F}{2\pi}} +1\bigg)}, \\
C&= C(d,f)=\frac{F}{\sqrt{2\pi}}+\sqrt{2\pi} B F^{\frac{1}{4}}.
\end{align*}
For $\zeta\sign(d)\ll -C^2 \mathbf{1}_{d<0}$ these bounds can be improved to  
$$
\|u\|_2 \leq D, \quad \|u\|_\infty \leq \biggl(\frac{F^{\frac{3}{4}}}{\sqrt{2\pi}}+\sqrt{2\pi}B\biggr)D^{\frac{1}{4}},
$$
where 
$$
D=D(d,f,\omega,\zeta)=\biggl(\frac{F^{\frac{3}{2}}+ |\omega|B F^{\frac{3}{4}} +|d|B^2}{(-\zeta \operatorname{sign}(d)-C^2 \mathbf{1}_{d<0})_+}\biggr)^{\frac{2}{3}}.
$$
\end{Theorem}

\begin{Remark} The improvement in the second part of the theorem lies in the fact that the bound $D$ becomes small when the detuning $\zeta$ is such that $\zeta \sign(d)$ is very negative.
\end{Remark}

\begin{proof} The proof is divided into five steps. 

\medskip

\noindent
Step 1. We first prove the $L^2$ estimate 
\begin{equation}\label{L2-bound}
\|u\|_2 \leq F= \|f\|_2.
\end{equation} To this end we multiply the differential equation \eqref{TWE}  with $\bar u$ to obtain
\begin{equation}\label{testedeq}
-d u'' \bar u+\mathrm{i} \omega u' \bar u+(\zeta-\mathrm{i})|u|^2-|u|^4+\mathrm{i}f \bar u=0.
\end{equation} 
Taking the imaginary part yields
\begin{equation}\label{IT}
-d \IT(u'' \bar u)+\omega \RT(u'\bar u)-|u|^2+\RT(f \bar u)=0.
\end{equation} 
Let $h\coloneqq |u|^2-\RT(f \bar u)$, $H\coloneqq-d\IT(u'\bar u)+\frac{\omega}{2}|u|^2$. Then $H'=h$ by equation \eqref{IT} and $H(0)=H(2\pi)$ by the periodicity of $u$. Hence
\begin{equation*}
0=H(2\pi)-H(0)=\int_0^{2 \pi} \, h \,ds =  \int_0^{2 \pi} \, |u|^2-\RT(f \bar u) \,ds
\end{equation*}
which implies
\begin{equation*}
\|u\|_2^2=\int_0^{2\pi} \, \RT(f \bar u) \, ds  \leq \|f\|_2 \|u\|_2  = F\|u\|_2.
\end{equation*}

\medskip

\noindent
Step 2. Next we prove 
\begin{equation}\label{L2leadstoL2foru'}
\|u'\|_2 \leq B\|u\|_2^{\frac{1}{4}}\leq B F^{\frac{1}{4}} .
\end{equation}
From \eqref{TWE} we may isolate the linear term $u$ and insert its derivative $u'$ into the following calculation for $\|u'\|_2^2$: 
\begin{align*}
\|u'\|_2^2& = \RT \int_0^{2\pi}\, u' \bar u'  \, ds \stackrel{\eqref{TWE}}{=}\RT \int_0^{2\pi} \, (\mathrm{i} d u''+\omega u'-\mathrm{i}\zeta u+\mathrm{i}|u|^2 u + f)' \bar u' \, ds  \\
&=\RT \int_0^{2\pi} \, \mathrm{i} d u'''\bar u'+\omega u'' \bar u' -\mathrm{i}\zeta |u'|^2+\mathrm{i}(|u|^2 u)'\bar u' + f'\bar u'  \, ds \\
&= \int_0^{2\pi} \, -d(\IT(u''\bar u'))'+\left(\frac{\omega}{2}|u'|^2\right)' \, ds  - \IT \int_0^{2\pi} \, (|u|^2 u)'\bar u' \, ds +\RT \int_0^{2\pi} \, f'\bar u'  \, ds \\
&= \int_0^{2\pi} \, (|u|^2)' \IT (\bar u u') -\RT(f'' \bar u)  \, ds +\RT f' \bar u \big\vert_0^{2\pi} \\
&\leq \int_0^{2\pi} \, \frac{1}{d}(|u|^2)' \Big(\frac{\omega}{2}|u|^2-H\Big)  +\|f''\|_2 \|u\|_2 +2\|f'\|_{\infty} \|u\|_{\infty} \\
&= \int_0^{2\pi} \, \frac{\omega}{4d} (|u|^4)'-\frac{1}{d}(|u|^2)' H +\|f''\|_2 \|u\|_2 +2\|f'\|_{\infty} \|u\|_{\infty} \\
&= \int_0^{2\pi} \, -\frac{1}{d}(|u|^2)'(H-H(0)) +\|f''\|_2 \|u\|_2 +2\|f'\|_{\infty} \|u\|_{\infty}.
\end{align*}
Next notice the pointwise estimate 
\begin{equation*}
h =|u|^2-\RT(f \bar u)  \geq |u|^2-|f||u| \geq - \frac{1}{4} |f|^2
\end{equation*}
from which we deduce the following two-sided estimate for $H-H(0)$:
\begin{align*}
H(s)-H(0)&=\int_0^s \, h(r) \, dr \geq -\frac{1}{4}\|f\|_2^2 \quad (s\in[0,2\pi]) \quad \text{ and } \\
H(s)-H(0)&=H(s)-H(2\pi)=-\int_s^{2\pi} \, h(r) \, dr \leq \frac{1}{4}\|f\|_2^2 \quad (s\in[0,2\pi]).
\end{align*}
Continuing the above inequality for $\|u'\|_2^2$ we conclude
\begin{equation*}
\|u'\|_2^2 \leq \frac{ \|f\|_2^2}{2|d|}  \|u\|_2 \|u'\|_2  +\|f''\|_2 \|u\|_2 +2\|f'\|_{\infty} \|u\|_{\infty}.
\end{equation*}
Next we want to get rid of the $\|u\|_{\infty}$ term. For that we note that there exists $s_0\in [0, 2\pi]$ satisfying $|u^2(s_0)| \leq \frac{1}{2\pi}\|u\|_2^2$. We use this in the following way,
\begin{align*}
\|u\|_{\infty}^2 &\leq |u^2(s_0)|+\sup_{s\in[0,2\pi]} |u^2(s)-u^2(s_0)|  \leq \frac{1}{2\pi} \|u\|_2^2+\int_0^{2\pi} \, 2|u| |u'| \, ds \\
&\leq \frac{1}{2\pi}\|u\|_2^2+2\|u\|_2\|u'\|_2 \stackrel{\eqref{L2-bound}}{\leq} \frac{F}{2\pi}\|u\|_2+2\|u\|_2\|u'\|_2 \\
&\leq \|u\|_2\bigg(\frac{F}{2\pi} +1+\|u'\|_2^2\bigg) ,
\end{align*}
from where we find
\begin{equation*}
\|u\|_{\infty} \leq \|u\|_2^{\frac{1}{2}} \bigg(\sqrt{\frac{F}{2\pi}} +1+\|u'\|_2\bigg).
\end{equation*}
In total, we have
\begin{align*}
&\|u'\|_2^2 \leq \frac{\|f\|_2^2}{2|d|}  \|u\|_2 \|u'\|_2  +\|f''\|_2 \|u\|_2 +2\|f'\|_{\infty} \|u\|_2^{\frac{1}{2}} \bigg(\sqrt{\frac{F}{2\pi}} +1+\|u'\|_2\bigg) \\
&\stackrel{\eqref{L2-bound}}{\leq} \frac{F^{\frac{11}{4}}}{2|d|}  \|u\|_2^{\frac{1}{4}} \|u'\|_2  +\|f''\|_2 F^{\frac{1}{2}} \|u\|_2^{\frac{1}{2}} +2\|f'\|_{\infty} \|u\|_2^{\frac{1}{2}} \bigg(\sqrt{\frac{F}{2\pi}} +1\bigg) +2\|f'\|_{\infty} F^{\frac{1}{4}} \|u\|_2^{\frac{1}{4}} \|u'\|_2 \\
&= \bigg(\frac{F^{\frac{11}{4}}}{2|d|}+2\|f'\|_{\infty} F^{\frac{1}{4}}\bigg)  \|u\|_2^{\frac{1}{4}} \|u'\|_2  +\bigg(\|f''\|_2 F^{\frac{1}{2}}+2\|f'\|_{\infty}\bigg(\sqrt{\frac{F}{2\pi}} +1\bigg)\bigg)\|u\|_2^{\frac{1}{2}} \\
&=:A_1\|u\|_2^{\frac{1}{4}} \|u'\|_2  +A_2^2 \|u\|_2^{\frac{1}{2}}. 
\end{align*}
This is a quadratic inequality in $\|u'\|_2$ which implies 
$$
\|u'\|_2 \leq \frac{A_1\|u\|_2^{\frac{1}{4}}+\sqrt{A_1^2\|u\|_2^{\frac{1}{2}}+4 A_2^2\|u\|_2^{\frac{1}{2}}}}{2} \leq A_1\|u\|_2^{\frac{1}{4}}+A_2 \|u\|_2^{\frac{1}{4}}=B \|u\|_2^{\frac{1}{4}}
$$
as claimed.

\medskip

\noindent
Step 3. Here we prove 
\begin{equation}\label{firstLinftybound}
\|u\|_{\infty}\leq C.
\end{equation}
There exists $s_1 \in [0,2\pi]$ satisfying $|u(s_1)|\leq \frac{\|u\|_2}{\sqrt{2\pi}}$. The claim now follows from
\begin{align*}
\|u\|_{\infty} \leq & |u(s_1)|+\sup_{s\in[0,2\pi]} |u(s)-u(s_1)| \leq \frac{\|u\|_2}{\sqrt{2\pi}}+\|u'\|_1 
\leq \frac{\|u\|_2}{\sqrt{2\pi}}+\sqrt{2\pi}\|u'\|_2  \\
\stackrel{\eqref{L2-bound}, \eqref{L2leadstoL2foru'}}{\leq} & \left(\frac{F^\frac{3}{4}}{\sqrt{2\pi}}+\sqrt{2\pi}B\right)\|u\|_2^\frac{1}{4}
\stackrel{\eqref{L2-bound}}{\leq} C.
\end{align*}

\medskip

\noindent
Step 4. Next we show in the case $\zeta \operatorname{sign}(d)<-C^2 \mathbf{1}_{d<0}$ the additional $L^2$-bound 
\begin{equation}\label{SecondL2bound}
\|u\|_2 \leq D.
\end{equation}
 After integrating \eqref{testedeq} over $[0,2\pi]$ and taking the real part of the resulting equation we get
\begin{equation*}
 d \|u'\|_2^2 =\omega\int_0^{2\pi} \, \IT(u'\bar u) \, ds -\zeta \|u\|_2^2 +\|u\|_4^4 +\IT \int_0^{2\pi} \, f \bar u\, ds.
\end{equation*}
In order to prove \eqref{SecondL2bound} we first suppose $d>0$. Then we have on one hand 
\begin{equation} \label{one_hand}
d \|u'\|_2^2 \stackrel{\eqref{L2leadstoL2foru'}}{\leq} d B^2 \|u\|_2^{\frac{1}{2}}
\end{equation}
and on the other hand 
\begin{equation} \label{other_hand}
\begin{split}
& \omega\int_0^{2\pi} \, \IT(u'\bar u) \, ds -\zeta \|u\|_2^2 +\|u\|_4^4  +\IT \int_0^{2\pi} \, f\bar u\, ds \\ 
&\geq  -|\omega| \|u\|_2\|u'\|_2 -\zeta \|u\|_2^2 -F\|u\|_2 \\
&\stackrel{\eqref{L2leadstoL2foru'}}{\geq} -|\omega|B\|u\|_2^{\frac{5}{4}}  -\zeta \|u\|_2^2 -F\|u\|_2 \\
&\stackrel{\eqref{L2-bound}}{\geq} -|\omega|B F^{\frac{3}{4}}\|u\|_2^{\frac{1}{2}} -\zeta \|u\|_2^2 -F^{\frac{3}{2}}\|u\|_2^{\frac{1}{2}}.
\end{split}
\end{equation}
Combining the two estimates \eqref{one_hand}, \eqref{other_hand} and grouping quadratic terms and terms of power $\frac{1}{2}$ of $\|u\|_2$ on separate sides of the inequality we get 
$$
-\zeta\|u\|_2^2 \leq \Bigl(F^{\frac{3}{2}}+ |\omega|B F^{\frac{3}{4}} +dB^2\Bigr)\|u\|_2^{\frac{1}{2}}
$$
which finally implies $\|u\|_2 \leq D$ whenever $\zeta<0$. Assuming now $d<0$ the estimate \eqref{one_hand} becomes 
\begin{equation} \label{one_hand2}
d \|u'\|_2^2 \geq -|d| B^2 \|u\|_2^{\frac{1}{2}}
\end{equation}
whereas in \eqref{other_hand} the term $\|u\|_4^4$, which was previously dropped, now has to be estimated by $\|u\|_4^4\leq \|u\|_\infty^2\|u\|_2^2\leq C^2\|u\|_2^2$. The estimate \eqref{other_hand} now becomes 
\begin{equation} \label{other_hand2}
\begin{split}
&\omega\int_0^{2\pi} \, \IT(u'\bar u) \, ds -\zeta \|u\|_2^2 +\|u\|_4^4  +\IT \int_0^{2\pi} \, f \bar u \, ds\\
&\leq |\omega|B F^{\frac{3}{4}}\|u\|_2^{\frac{1}{2}} +(C^2-\zeta) \|u\|_2^2 +F^{\frac{3}{2}}\|u\|_2^{\frac{1}{2}}.
\end{split}
\end{equation} 
The combination of \eqref{one_hand2} and \eqref{other_hand2} leads to 
$$ 
(\zeta-C^2)\|u\|_2^2 \leq \Bigl(F^{\frac{3}{2}}+ |\omega|B F^{\frac{3}{4}} +|d|B^2\Bigr)\|u\|_2^{\frac{1}{2}}
$$
which again implies $\|u\|_2 \leq D$ whenever $-\zeta<-C^2$. 

\medskip

\noindent
Step 5. Finally we prove 
\begin{equation}\label{secondLinftybound}
\|u\|_{\infty}\leq  \biggl(\frac{F^{\frac{3}{4}}}{\sqrt{2\pi}}+\sqrt{2\pi}B\biggr)D^{\frac{1}{4}}
\end{equation}
whenever $\zeta \operatorname{sign}(d)<-C^2 \mathbf{1}_{d<0}$. For this we repeat Step 3 and use in the final estimate that $\|u\|_2\leq D$.
\end{proof}

\section{Proof of existence (Theorem~\ref{Existence}) and uniqueness (Theorem~\ref{Uniqueness}) statements} \label{sec:existence_and_uniqueness}

Let us consider the operator $L:H^2_{\text{per}}(0,2\pi) \to L^2(0,2\pi)$ with $Lu=L_0u -\iu u$ and $L_0 u=-d u''+\iu\omega u'+\zeta u$. Since $L_0:H^2_{\text{per}}(0,2\pi) \to L^2(0,2\pi)$ is self-adjoint its spectrum is real and we see that $L$ has spectrum on the line $-\iu+\R$. In particular, $L$ is invertible and $L^{-1}:L^2(0,2\pi)\to H^2_{\text{per}}(0,2\pi)$ is bounded. By using the compact embedding $H^2_{\text{per}}(0,2\pi)  \hookrightarrow H^1_{\text{per}}(0,2\pi)$ we see that 
$$
L^{-1}:L^2(0,2\pi) \to H^1_{\text{per}}(0,2\pi) \mbox{ is compact}.
$$ 
Since moreover $H^1_{\text{per}}(0,2\pi)$ is a Banach algebra we can rewrite \eqref{TWE} as a fixed point problem $u=\Phi(u)$, where $\Phi$ denotes the compact map 
\begin{equation*}
\Phi:H^1_{\text{per}}(0,2\pi) \to H^1_{\text{per}}(0,2\pi), \ \Phi(u)=L^{-1}\big(|u|^2 u-\mathrm{i}f(s)\big).
\end{equation*}
In order to prove our first existence result from Theorem~\ref{Existence}, let us recall Schaefer's fixed point theorem (\cite[Corollary 8.1]{Deimling}).

\begin{Theorem}[Schaefer's fixed point theorem]\label{Schaefer}
Let $X$ be a Banach space and $\Phi:X \to X$ be compact. Suppose that the set 
\begin{equation*}
\{x \in X: \, x=\lambda \Phi(x) \text{ for some } \lambda \in(0,1)\}
\end{equation*}
 is bounded. Then $\Phi$ has a fixed point. 
\end{Theorem}

\begin{proof}[Proof of Theorem~\ref{Existence}]
Let $u\in H^1_{\text{per}}(0,2\pi)$ and $u=\lambda \Phi(u)$ for some $\lambda \in (0,1)$. Then $u\in H^2_{\text{per}}(0,2\pi)$ and 
\begin{equation*}
-d u''+\mathrm{i} \omega u'+(\zeta-\mathrm{i})u-\lambda |u|^2 u+\mathrm{i}\lambda f(s)=0.
\end{equation*}
Let us now define $v \in H^2_{\text{per}}(0,2\pi)$ by $v(s)=\sqrt{\lambda}u(s)$. Then 
\begin{equation*}-d v''+\mathrm{i} \omega v'+(\zeta-\mathrm{i})v-|v|^2 v+\mathrm{i} \tilde f(s)=0
\end{equation*}
with $\tilde f=\lambda^\frac{3}{2} f$. Estimate \eqref{l2} of Theorem~\ref{a-priori} with $\tilde F=F(\lambda^\frac{3}{2} f)=\lambda^\frac{3}{2}F$ implies
\begin{equation*}
\|u\|_2 = \frac{1}{\sqrt{\lambda}} \|v\|_2 \leq \frac{1}{\sqrt{\lambda}} \tilde F=\lambda F\leq F.
\end{equation*}
Using \eqref{l2_prime} from Theorem~\ref{a-priori} with $\tilde B=B(d,\lambda^\frac{3}{2} f)$ we also find 
\begin{align*}
\|u'\|_2 &= \frac{1}{\sqrt{\lambda}} \|v'\|_2 \leq  \frac{1}{\sqrt{\lambda}}\tilde B \tilde F^{\frac{1}{4}} \\
& = \lambda^4\frac{F^3}{2|d|}+2\lambda^{\frac{7}{4}}\|f'\|_{\infty} F^{\frac{1}{2}}+\sqrt{\lambda^2 \|f''\|_2 F+2\lambda^{\frac{5}{4}}\|f'\|_{\infty} \biggl(\frac{\lambda^{\frac{3}{4}} F}{\sqrt{2\pi}}+\sqrt{F}\biggr)} \\
 &\leq \frac{F^3}{2|d|}+2\|f'\|_{\infty} F^{\frac{1}{2}}+\sqrt{\|f''\|_2 F+2\|f'\|_{\infty} \biggl(\frac{F}{\sqrt{2\pi}}+\sqrt{F}\biggr)}=B F^{\frac{1}{4}}.
\end{align*}
The assertion now follows from Theorem \ref{Schaefer}.
\end{proof}

For the next uniqueness result, cf. Theorem~\ref{Uniqueness}, let us rewrite the constant $D$ from Theorem~\ref{a-priori} as 
$$
D=D(d,f,\omega,\zeta)=\biggl(\frac{\tilde D}{(-\zeta \operatorname{sign}(d)-C^2 \mathbf{1}_{d<0})_+}\biggr)^{\frac{2}{3}}
$$
with 
$$
\tilde D=\tilde D(d,f,\omega) = F^{\frac{3}{2}}+ |\omega|B F^{\frac{3}{4}} +|d|B^2.
$$
Our result complements the existence statement provided in Theorem~\ref{Existence} by a uniqueness statement. It consists of three cases: (i) and (ii) cover the case where $|\zeta|\gg 1$ is sufficiently large whereas (iii) builds upon $\|f\|\ll 1$ measured in a suitable norm  $\|\cdot\|$ such that the constant $C=C(d,f)$ becomes small. This is the case, e.g., if $\|f\|_2\ll 1$ and $\|f''\|_2$ remains bounded.

\begin{Theorem}\label{Uniqueness}
Let $d \in \R\setminus\{0\}$, $\zeta,\omega\in\R$ and $f \in H^2(0,2\pi)$. Then \eqref{TWE} has a unique solution $u \in H^2_{\text{per}}(0,2\pi)$ in the following three cases,
\begin{itemize}
\item[(i)] $$\sign(d)\zeta <\zeta_{*},$$
\item[(ii)] $$\sign(d)\zeta >\zeta^*,$$
\item[(iii)] $$\sqrt{3}C<1,$$
\end{itemize}
where $\zeta_{*}\leq 0 \leq \zeta^*$ are given by
\begin{align*}
\zeta_{*}&=\zeta_{*}(d,f,\omega)=-C^2 \mathbf{1}_{d<0} -\frac{27(F^{\frac{3}{4}}+2\pi B)^6 \tilde D}{8\pi^3}, \\
\zeta^*&=\zeta^*(d,f,\omega)=3C^2+\frac{\omega^2}{4|d|},
\end{align*}
and $F=F(f)$, $B=B(d,f)$, $C=C(d,f)$ are the constants from Theorem~\ref{a-priori}. 
\end{Theorem}

\begin{proof}
It suffices to consider the case $f \neq 0$. By Theorem~\ref{Existence} we know that \eqref{TWE} has at least one solution $u_1\in H^2_{\text{per}}(0,2\pi)$. Now let $u_2\in H^2_{\text{per}}(0,2\pi)$ denote an additional solution and define
\begin{equation*}
R=R(d,f,\omega,\zeta)=\begin{cases} \min\biggl\{C,\biggl(\frac{F^{\frac{3}{4}}}{\sqrt{2\pi}}+\sqrt{2\pi}B\biggr)D^{\frac{1}{4}}\biggr\}, & \zeta\operatorname{sign}(d)+C^2 \mathbf{1}_{d<0}<0, \\ C, & \zeta\operatorname{sign}(d)+C^2 \mathbf{1}_{d<0}\geq 0. \end{cases}
\end{equation*}
Then $\|u_j\|_{\infty}\leq R$ for $j=1,2$ by Theorem~\ref{a-priori}, which easily implies
\begin{equation*}
\big\||u_1|^2 u_1-|u_2|^2u_2\big\|_2 \leq 3R^2 \|u_1-u_2\|_2.
\end{equation*} 
Since $u_j, j=1,2$ solves the fixed point problem $u_j=\Phi(u_j)$ we obtain 
\begin{equation*} 
\|u_1-u_2\|_2=\|\Phi(u_1)-\Phi(u_2)\|_2 \leq 3R^2\|L^{-1}\|\|u_1-u_2\|_2,
\end{equation*}
where $\|L^{-1}\|=\sup_{v \in L^2(0,2\pi), \|v\|_2=1} \|L^{-1}v\|_2 $. 
Next we show $3R^2\|L^{-1}\|<1$ which implies $u_1=u_2$ and thus finishes the proof. To this end we decompose a function $v\in L^2(0,2\pi)$ into its Fourier series, i.e., $v=\sum_{m\in\Z} v_m \mathrm{e}^{\mathrm{i}ms}$ so that
\begin{equation*}
L^{-1}v=\sum\limits_{m\in\Z} \frac{v_m}{dm^2-\omega m+\zeta-\mathrm{i}}\, \mathrm{e}^{\mathrm{i}ms}.
\end{equation*}
On one hand we get $\|L^{-1}\| \leq 1$ since
\begin{equation*}
\|L^{-1}v\|_2^2=2\pi \sum\limits_{m\in\Z} \frac{|v_m|^2}{1+(d m^2-\omega m+\zeta)^2} \leq 2\pi \sum\limits_{m\in\Z}|v_m|^2=\|v\|_2^2.
\end{equation*}
On the other hand, if $\operatorname{sign}(d)\big(\zeta-\frac{\omega^2}{4d}\big)>0$, we get 
\begin{align*}
\|L^{-1}v\|_2^2&=2\pi \sum\limits_{m\in\Z} \frac{|v_m|^2}{1+(d m^2-\omega m+\zeta)^2} =2\pi \sum\limits_{m\in\Z} \frac{|v_m|^2}{1+\Big(d\big(m-\frac{\omega}{2d}\big)^2+\zeta-\frac{\omega^2}{4d}\Big)^2} \\
&\leq 2\pi \sum\limits_{m\in\Z} \frac{|v_m|^2}{\big(\zeta-\frac{\omega^2}{4d}\big)^2}=\frac{1}{\big(\zeta-\frac{\omega^2}{4d}\big)^2} \|v\|_2^2,
\end{align*}
i.e. $\|L^{-1}\| \leq \operatorname{sign}(d)\big(\zeta-\frac{\omega^2}{4d}\big)^{-1}$.

\medskip

In case (i) where $\operatorname{sign}(d)\zeta <\zeta_*<-C^2 \mathbf{1}_{d<0} \leq \textcolor{black}{0}$ we use $\|L^{-1}\|\leq 1$ and find by the definition of $R$ and $\zeta_*$ that
\begin{align*}
3R^2 \|L^{-1}\| &\leq 3 \frac{(F^{\frac{3}{4}}+2\pi B)^2}{2\pi} D^{\frac{1}{2}} \\
&=  3 \frac{(F^{\frac{3}{4}}+2\pi B)^2}{2\pi} \biggl(\frac{\tilde D}{-\zeta \operatorname{sign}(d)-C^2 \mathbf{1}_{d<0}}\biggr)^{\frac{1}{3}}  \\ 
&< 3 \frac{(F^{\frac{3}{4}}+2\pi B)^2}{2\pi} \biggl(\frac{\tilde D}{-\zeta_* -C^2 \mathbf{1}_{d<0}}\biggr)^{\frac{1}{3}}=1.
\end{align*}

In case (ii) where $\operatorname{sign}(d)\zeta >\zeta^* > \frac{\omega^2}{4|d|}\geq 0$ we use $\|L^{-1}\| \leq \operatorname{sign}(d)\big(\zeta-\frac{\omega^2}{4d}\big)^{-1}$ and get by the choice of $\zeta^*$
\begin{equation*}
3R^2 \|L^{-1}\| \leq \frac{3 C^2}{\operatorname{sign}(d)(\zeta-\frac{\omega^2}{4d})} <\frac{3C^2}{\zeta^*-\frac{\omega^2}{4|d|}}=1.
\end{equation*}

In case (iii) where $\sqrt{3}C<1$ we use $\|L^{-1}\|\leq 1$ to conclude
\begin{equation*}
3R^2 \|L^{-1}\| \leq 3 C^2 <1.
\end{equation*}
\end{proof}

\section{Proof of the continuation results} \label{sec:continuation}

In this section we continue to use the notion for the operator $L:H^2_{\text{per}}(0,2\pi) \to L^2(0,2\pi)$ from Section~\ref{sec:a-priori}. We also use that $L^{-1}:L^2(0,2\pi) \to H^2_{\text{per}}(0,2\pi)$ is bounded and that $L^{-1}:L^2(0,2\pi) \to H^1_{\text{per}}(0,2\pi)$ is compact. We first consider continuation from a trivial solution. In order to prove Theorem~\ref{Hauptresultat} let us provide the following global continuation theorem.

\begin{Theorem}\label{GCT}
Let $X$ be a real Banach space and $K \in C^1(\R\times X,X)$ be compact. We consider the problem 
\begin{equation}\label{CPOI}
T(\lambda,x)\coloneqq x-K(\lambda,x)=0.
\end{equation}
Assume that $T(\lambda_0,x_0)=0$ and that $\partial_x T(\lambda_0,x_0)$ is invertible. Then there exists a connected and closed set (=continuum) $\mathcal{C}^+ \subset [\lambda_0,\infty) \times X$ of solutions of \eqref{CPOI} with $(\lambda_0,x_0)\in \mathcal{C}^+$. For $\mathcal{C}^+$ one of the following alternatives holds:
\begin{enumerate}[(a)]
\item $\mathcal{C}^+$ is unbounded,
\end{enumerate}
or 
\begin{enumerate}[(b)]
\item $\exists x_0^+ \in  X\setminus\{x_0\}: \, (\lambda_0,x_0^+) \in \mathcal{C}^+.$
\end{enumerate}
If one chooses $\mathcal{C}^+$ to be maximally connected then there is no more a strict alternative between (a) and (b) and instead at least one of the two (possibly both) properties holds.
\end{Theorem}

\begin{Remark}\label{RemarkHauptresultat}
($\alpha$) The theorem follows from \cite[Theorem 3.3]{Bandle} or \cite[Theorem 1.3.2]{Schmitt} since $\operatorname{deg}(T(\lambda_0,\cdot),B_{\varepsilon}(x_0),0)=\operatorname{deg}(\partial_x T(\lambda_0,x_0),B_{\varepsilon}(0),0) \neq 0$ because $\partial_x T(\lambda_0,x_0)$ is invertible. \\
($\beta$) There exists also a continuum $\mathcal{C}^- \subset (-\infty,\lambda_0] \times X$ of solutions of \eqref{CPOI} with $(\lambda_0,x_0)\in \mathcal{C}^-$ satisfying one of the alternatives of the theorem. \\
($\gamma$) Alternative (a) of Theorem~\ref{GCT} means that $\mathcal{C}^+$ is unbounded either in the Banach space direction $X$ or in the parameter direction $[\lambda_0,\infty)$ or in both. If unboundedness in the Banach space direction is excluded on compact intervals $[\lambda_0,\Lambda]$, e.g., by a-priori bounds, then unboundedness in the parameter direction follows, i.e., the projection of $\mathcal{C}^+$ onto $[\lambda_0,\infty)$ denoted by $\operatorname{pr}_1(\mathcal{C}^+)$ must coincide with $[\lambda_0,\infty)$. This is an existence result for all $\lambda\geq\lambda_0$ which is one aspect of Theorem \ref{Hauptresultat}.\\
($\delta$) Alternative (b) of Theorem~\ref{GCT} means that the continuum $\mathcal{C}^+$ returns to the $\lambda=\lambda_0$ line at a point $x_0^+\not = x_0$.
\end{Remark}

\begin{proof}[Proof of Theorem~\ref{Hauptresultat}] Let $K:\R\times H^1_{\text{per}}(0,2\pi) \to H^1_{\text{per}}(0,2\pi), \, K(f_1,u)\coloneqq L^{-1}(|u|^2 u-\mathrm{i}f_0-\mathrm{i}f_1 e(s))$ and $T(f_1,u)\coloneqq u-K(f_1,u)$. Then, as explained before Theorem~\ref{Schaefer}, $K$ is compact and 
\begin{equation*}
T(0,u_0)=u_0-L^{-1}(|u_0|^2 u_0-\mathrm{i}f_0)\stackrel{\eqref{trivial}}{=}u_0-L^{-1}\big((\zeta-\mathrm{i})u_0\big)=u_0-u_0=0.
\end{equation*}
Next we show that $\partial_u T(0,u_0)$ is invertible. To this end note that 
\begin{equation*}
\partial_u T(0,u_0) \varphi =\varphi -L^{-1} (2|u_0|^2 \varphi+u_0^2 \overline{\varphi}) \text{ for } \varphi \in H^1_{\text{per}}(0,2\pi)
\end{equation*}
and hence, as a compact perturbation of the identity, $\partial_u T(0,u_0)$ is invertible if it is injective. Since $u_0$ is constant this amounts exactly to the characterization of non-degeneracy of $u_0$ as described in Lemma~\ref{charactrization_nondeg}.

\medskip

Now assertion (i) follows from the classical implicit function theorem and Theorem~\ref{GCT} yields that the maximal continuum $\mathcal{C}^+ \subset [0,\infty) \times H^1_{\text{per}}(0,2\pi)$ of solutions $(f_1,u)$ of \eqref{TWE} with $(0,u_0)\in\mathcal{C}^+$  is unbounded or returns to another solution at $f_1=0$. The continuum $\mathcal{C}^+$ in fact belongs to $[0,\infty)\times H^2_{\text{per}}(0,2\pi)$ and persists as a connected and closed set in the stronger topology of $[0,\infty)\times H^2_{\text{per}}(0,2\pi)$. Next we show that the unboundedness of $\mathcal{C}^+$ coincides with $\operatorname{pr}_1(\mathcal{C}^+)=[0,\infty)$. According to Remark~\ref{RemarkHauptresultat}.($\gamma$) we need to show that unboundedness in the Banach space direction $H^1_\text{per}(0,2\pi)$ is excluded for $f_1$ in bounded intervals. To see this suppose that $0\leq f_1 \leq M$ for all $(f_1,u)\in \mathcal{C}^+$ and some constant $M>0$. Then, by the a-priori bounds \eqref{l2} and \eqref{l2_prime} from Theorem~\ref{a-priori} we get
\begin{equation*}
\|u\|_2 \leq \|f_0+f_1 e(s) \|_2 \leq \sqrt{2\pi}|f_0|+M\|e\|_2=:N=N(f_0,M,e)
\end{equation*}
and 
\begin{equation*}
\|u'\|_2 \leq \frac{N^3}{2|d|}+2M \|e'\|_{\infty} N^{\frac{1}{2}} +\sqrt{M\|e''\|_2 N+2M\|e'\|_{\infty}\biggl(\frac{N}{\sqrt{2\pi}}+\sqrt{N}\biggr)}
\end{equation*}
for all $(f_1,u)\in \mathcal{C}^+$. Hence $\mathcal{C}^+$ is bounded in the Banach space direction. Assertion (ii) follows in a similar way by using the a-priori bounds of Theorem~\ref{a-priori} and the fact that by \eqref{TWE} the bounds for $\|u\|_2$, $\|u'\|_2$ and $\|u\|_{\infty}$ translate into a bound for $\|u''\|_2$.

According to Remark~\ref{RemarkHauptresultat}.($\beta$) the above line of arguments also yield that the maximal continuum $\mathcal{C}^- \subset (-\infty,0] \times  H^2_{\text{per}}(0,2\pi)$ of solutions of \eqref{TWE} with $(0,u_0)\in\mathcal{C}^-$ satisfies $\operatorname{pr}_1(\mathcal{C}^-)=(-\infty,0]$ or returns to another solution at $f_1=0$. This finishes the proof.
\end{proof}

\begin{proof}[Proof of Corollary~\ref{Korollar}] The result follows from a combination of Theorem~\ref{Hauptresultat} and Theorem~\ref{Uniqueness}. For $f_1=0$, i.e. $f(s)= f_0$, the abbreviations $F,B,C$ from Theorem~\ref{a-priori} and $\tilde D$ from Theorem~\ref{Uniqueness} reduce to  
\begin{align*}
F(f_0)&=\sqrt{2\pi}|f_0|, \quad B(d,f_0)=2^{\frac{3}{8}} \pi^{\frac{11}{8}} |f_0|^{\frac{11}{4}} |d|^{-1}, \\
C(d,f_0)& =|f_0|(1+2\pi^2 f_0^2 |d|^{-1}), \\
\tilde D(d,f_0,\omega) &=(2\pi)^{\frac{3}{4}}|f_0|^{\frac{3}{2}} (|d|+\pi f_0^2 |\omega| +\pi^2 f_0^4)|d|^{-1}. \\
\end{align*}
Hence the constants $\zeta_*, \zeta^*$ from Theorem~\ref{Uniqueness} take the form
\begin{align*}
\zeta_{*}(d,f_0,\omega)&=-C^2(d,f_0) \mathbf{1}_{d<0}- 27\biggl(1+\frac{\pi f_0^2 |\omega|}{|d|}+\frac{\pi^2 f_0^4}{|d|}\biggr)  C(d,f_0)^6, \\
\zeta^*(d,f_0,\omega)&=  3C(d,f_0)^2+\frac{\omega^2}{4|d|}.
\end{align*}
Finally, the conditions (i), (ii), (iii) from the uniqueness result of Theorem~\ref{Uniqueness} translate into the conditions (i), (ii), (iii) from Corollary~\ref{Korollar}. 
\end{proof}

Now we turn to continuation from a non-trivial solution. Theorem~\ref{Fortsetzung_nichttrivial} will follow from the Crandall-Rabinowitz Theorem of bifurcation from a simple eigenvalue, which we recall next. 

\begin{Theorem}[Crandall-Rabinowitz \cite{CrRab_bifurcation},\cite{kielhoefer2011bifurcation}] \label{Thm Crandall-Rabinowitz}
    Let $I\subset\R$ be an open interval, $X$,$Y$ Banach spaces and let $F:I\times X\to Y$ be twice continuously differentiable such
    that $F(\lambda,0)=0$ for all $\lambda\in I$ and $\partial_x F(\lambda_0,0):X\to Y$ is an index-zero Fredholm
    operator for $\lambda_0\in I$. Moreover assume:
    \begin{itemize}
      \item[(H1)] there is $\phi\in X,\phi\neq 0$ such that $\ker\partial_x F(\lambda_0,0)=\spann\{\phi\}$,
      \item[(H2)] $\partial^2_{x,\lambda} F(\lambda_0,0)[\phi] \not \in \range \partial_xF(\lambda_0,0)$.
    \end{itemize}
    Then there exists $\epsilon>0$ and a continuously differentiable curve $(\lambda,x): (-\epsilon,\epsilon)\to I\times X$ with $\lambda(0)=\lambda_0$, $x(0)=0$, $\dot x(0)=\phi$ and $x(t)\neq 0$ for $0<|t|<\epsilon$ and $F(\lambda(t),x(t))=0$ for all $t\in (-\epsilon,\epsilon)$. Moreover, there exists a neighborhood $J\times U\subset I\times X$ of $(\lambda_0,0)$ such that all non-trivial solutions in $J\times U$ of $F(\lambda,x)=0$ lie on the curve. Finally,
    $$
   \dot \lambda(0)= -\frac{1}{2} \frac{\langle \partial^2_{xx} F(\lambda_0,0)[\phi,\phi],\phi^*\rangle}{\langle \partial^2_{x,\lambda} F(\lambda_0,0)[\phi],\phi^*\rangle},
    $$
    where $\spann\{\phi^*\}=\ker \partial_x F(\lambda_0,0)^*$ and $\langle\cdot,\cdot\rangle$ is the duality pairing between $Y$ and its dual $Y^*$.  
    \end{Theorem}
    
Next we provide the functional analytic setup. Fix the values of $d, \omega, \zeta, f_0$ and the function $e$. If $u_0\in H^2_\text{per}(0,2\pi)$ is the non-trivial non-degenerate solution of \eqref{TWE} for $f_1=0$ (as assumed in Theorem~\ref{Fortsetzung_nichttrivial}) then for $\sigma\in\R$ we denote by $u_\sigma(s)\coloneqq u_0(s-\sigma)$ its shifted copy, which is also a solution of \eqref{TWE} for $f_1=0$. Consider the mapping
$$
G:\left\{\begin{array}{rcl}
\R \times H^2_{\text{per}(0,2\pi)} & \to & L^2(0,2\pi), \vspace{\jot} \\
(f_1,u) & \mapsto & -d u''+\mathrm{i} \omega u'+(\zeta-\mathrm{i})u-|u|^2 u+\mathrm{i}f_0+\mathrm{i}f_1 e(s).
\end{array}
\right.
$$
Then $G$ is twice continuously differentiable. The linearized operator $\partial_{(f_1,u)} G(0,u_\sigma)=(\iu e,L_{u_\sigma})$ with $L_{u_\sigma}$ as in Definition~\ref{non_degenerate_solution} is a Fredholm operator and $(0,u_\sigma')\in \ker \partial_{(f_1,u)} G(0,u_\sigma)$. As we shall see there may be more elements in the kernel. Next we fix the value $\sigma_0$ (its precise value will be given later) and let $H^2_\text{per}(0,2\pi) = \spann\{u_{\sigma_0}'\}\oplus Z$ where, e.g.,  
$$
Z \coloneqq H^2_\text{per}(0,2\pi) \cap\spann\{u_{\sigma_0}'\}^{\perp_{L^2}} =\biggl\{\varphi-\frac{\langle \varphi, u_{\sigma_0}'\rangle_{L^2}}{\langle u_{\sigma_0}', u_{\sigma_0}'\rangle_{L^2}} u_{\sigma_0}': \varphi\in H^2_\text{per}(0,2\pi)\biggr\}.
$$
It will be more convenient to rewrite $u= u_\sigma+v$ with $v\in Z$. In order to justify this, note also that the map $(\sigma,v)\mapsto u_{\sigma}+v$ defines a diffeomorphism of a neighborhood of $(\sigma_0,0) \in \R \times Z$ onto a neighborhood of $u_{\sigma_0} \in H^2_{\text{per}}(0,2\pi)$ since the derivative at $(\sigma_0,0)$ is given by $(\lambda,\psi) \mapsto -\lambda u_{\sigma_0}'+\psi$ which is an isomorphism from $\R \times Z$ onto $H^2_{\text{per}}(0,2\pi)$. Now we define 
$$
F:\left\{\begin{array}{rcl}
\R\times \R\times Z \ & \to & L^2(0,2\pi), \vspace{\jot} \\
(\sigma,f_1,v) & \mapsto & G(f_1,u_\sigma+v)
\end{array}
\right.
$$
which is also twice continuously differentiable and where $\partial_{(f_1,v)} F(\sigma_0,0,0)$ is an index-zero Fredholm operator. Our goal will be to solve 
\begin{equation}
\label{bifurcation_s}
F(\sigma,f_1,v)=0
\end{equation}
by means of bifurcation theory, where $\sigma\in\R$ is the bifurcation parameter. Notice that $F(\sigma,0,0)=0$ for all $\sigma\in \R$, i.e., $(f_1,v)=(0,0)$ is a trivial solution of \eqref{bifurcation_s}.

\medskip

Next we show (H1) of Theorem~\ref{Thm Crandall-Rabinowitz}.

\begin{Lemma} Suppose that $\sigma_0 \in\R$ satisfies \eqref{eq:sigma_0}, i.e. $\IT \int_0^{2\pi}e(s+\sigma_0)\overline{\phi_0^\ast(s)}\,ds =0$. Then  $\dim\ker\partial_{(f_1,v)} F(\sigma_0,0,0)=1$ and $\range\partial_{(f_1,v)} F(\sigma_0,0,0) = \spann\{\phi_{\sigma_0}^*\}^{\perp_{L^2}}$. 
\label{one_d_kernel}
\end{Lemma}

\begin{proof} The fact that $\partial_{(f_1,v)} F(\sigma_0,0,0)$ is a Fredholm operator follows from Remark~\ref{fredholm_etc}. For $(\alpha,\psi_{\sigma_0})\in \R\times Z$ being non-trivial and belonging to the kernel of $\partial_{(f_1,v)} F(\sigma_0,0,0)$ we have
\begin{equation} \label{eq:one_d_kernel}
\partial_{(f_1,v)} F(\sigma_0,0,0)[\alpha,\psi_{\sigma_0}]= L_{u_{\sigma_0}}\psi_{\sigma_0}+\mathrm{i}\alpha e=0.
\end{equation}
If $\alpha=0$ then by non-degeneracy we find $\psi_{\sigma_0}\in \spann\{u_{\sigma_0}'\}\cap Z=\{0\}$, which is impossible. Hence we may assume w.l.o.g. that $\alpha=1$ and $\psi_{\sigma_0}$ has to solve 
\begin{equation} \label{def_psi}
L_{u_{\sigma_0}}\psi_{\sigma_0} = -\mathrm{i}e
\end{equation}
which, by setting $\psi_{\sigma_0}(s)=\xi_{\sigma_0}(s-\sigma_0)$, is equivalent to
\begin{equation} \label{def_xi}
L_{u_0}\xi_{\sigma_0} = -\mathrm{i}e(\cdot+\sigma_0).
\end{equation}
By the Fredholm alternative this is possible if and only if $-\iu e(\cdot+\sigma_0)\perp_{L^2} \phi_0^*$. If this $L^2$-ortho\-gonality holds then there exists $\psi_{\sigma_0}\in H^2_\text{per}(0,2\pi)$ solving \eqref{def_psi} and $\psi_{\sigma_0}$ is unique up to adding a multiple of $u_{\sigma_0}'$. Hence there is a unique $\psi_{\sigma_0}\in Z$ solving \eqref{def_psi}. The $L^2$-orthogonality means 
\begin{equation*}
0=-\RT\int_0^{2\pi}\iu e(s+\sigma_0) \overline{\phi_0^*(s)}\,ds = \IT \int_0^{2\pi} e(s+\sigma_0) \overline{\phi_0^*(s)}\,ds
\end{equation*}
which amounts to \eqref{eq:sigma_0}. Finally, it remains to determine the range of $\partial_{(f_1,v)} F(\sigma_0,0,0)$. Let $\tilde\phi\in L^2(0,2\pi)$ be such that $\tilde\phi = \partial_{(f_1,v)} F(\sigma_0,0,0)[\alpha,\tilde\psi]$ with $\tilde\psi\in Z$ and $\alpha \in \R$. Thus 
\begin{equation} \label{eq:range}
L_{u_{\sigma_0}}\tilde\psi+\mathrm{i}\alpha e=\tilde\phi
\end{equation}
and since $\mathrm{i}e\perp_{L^2} \phi_{\sigma_0}^*$ by the definition of $\sigma_0$, the Fredholm alternative says that a necessary and sufficient condition for $\tilde\phi$ to satisfy \eqref{eq:range} is that $\tilde\phi\in \spann\{\phi_{\sigma_0}^*\}^{\perp_{L^2}}$ as claimed. Note that in this case $\tilde\psi\in H^2_\text{per}(0,2\pi)=\ker L_{u_{\sigma_0}}\oplus Z$ and hence, for every given $\alpha\in \R$ and $\tilde\phi\in\spann\{\phi_{\sigma_0}^*\}^{\perp_{L^2}}$ there is a unique element $\tilde\psi\in Z$ that solves \eqref{eq:range}.
\end{proof}

\begin{proof}[Proof of Theorem~\ref{Fortsetzung_nichttrivial}.] The proof is divided into three steps.

\medskip

\noindent
Step 1. We begin by verifying for \eqref{bifurcation_s} the conditions for the local bifurcation theorem of Crandall-Rabinowitz, cf. Theorem~\ref{Thm Crandall-Rabinowitz}. By Lemma~\ref{one_d_kernel}, $\partial_{(f_1,v)} F(\sigma_0,0,0): \R\times Z\to L^2(0,2\pi)$ is an index-zero Fredholm operator and it satisfies 
$$
\ker\partial_{(f_1,v)} F(\sigma_0,0,0)=\spann\{(1,\psi_{\sigma_0})\},
$$
where $\psi_{\sigma_0}$ denotes the unique element of Z which solves \eqref{def_psi}. Hence (H1) is satisfied. To see (H2) note that 
$$
\partial^2_{(f_1,v),\sigma} F(\sigma_0,0,0)[1,\psi_{\sigma_0}]=2u_{\sigma_0}'\overline{u_{\sigma_0}}\psi_{\sigma_0}+2\overline{u_{\sigma_0}'} u_{\sigma_0} \psi_{\sigma_0}+2u_{\sigma_0}u_{\sigma_0}'\overline{\psi_{\sigma_0}}.
$$
On the other hand, differentiation of \eqref{def_psi} w.r.t. $s$ yields
\begin{equation} \label{imageofpsi'}
L_{u_{\sigma_0}}\psi'_{\sigma_0} = 2u_{\sigma_0}'\overline{u_{\sigma_0}}\psi_{\sigma_0}+2\overline{u_{\sigma_0}'} u_{\sigma_0} \psi_{\sigma_0} +2u_{\sigma_0}u_{\sigma_0}'\overline{\psi_{\sigma_0}}-\iu e'
\end{equation}
so that 
\begin{equation} \label{transv_imageofpsi'}
\partial^2_{(f_1,v),\sigma} F(\sigma_0,0,0)[1,\psi_{\sigma_0}] = L_{u_{\sigma_0}}\psi_{\sigma_0}'+\iu e'.
\end{equation}
Hence the characterization of $\range\partial_{(f_1,v)} F(\sigma_0,0,0)$ from Lemma~\ref{one_d_kernel} implies that the trans\-versality condition (H2) is satisfied if and only if $\RT\int_0^{2\pi} \iu e'(s)\overline{\phi_{\sigma_0}^*(s)}\,ds \not=0$ which amounts to assumption \eqref{eq:transv}. This already allows us to apply Theorem~\ref{Thm Crandall-Rabinowitz} and we obtain the existence of a local curve $t \mapsto (\sigma(t), f_1(t),v(t))$, $\dot{f_1}(0)=1$, $f_1(0)=0$, $v(0)=0$, $\sigma(0)=\sigma_0$ with $F(\sigma(t), f_1(t),v(t))=0$. Assertion (i) is then satisfied with $u(t) \coloneqq u_{\sigma(t)}+v(t)$. Assertion (ii) follows like in the proof of Theorem~\ref{Hauptresultat}.

\medskip

\noindent
Step 2. From here on let us additionally assume that zero is an algebraically simple eigenvalue of $L_{u_0}$, i.e. $u_0' \notin \range L_{u_0}$. Next we want to show that $L_{u(t)}$ is invertible for $0<|t|<\delta^*$ and $\delta^*$ sufficiently small, i.e. that the critical zero eigenvalue of $L_{u(0)}=L_{u_{\sigma_0}}$ moves away from zero when $t$ evolves. Let us define
$$
H:\left\{\begin{array}{rcl}
H^2_{\text{per}}(0,2\pi) \times Z \times \R \ & \to & L^2(0,2\pi), \vspace{\jot} \\
(u,v,\mu) & \mapsto & L_u (u_{\sigma_0}'+v) -\mu (u_{\sigma_0}'+v).
\end{array}
\right.
$$
Then $H(u_{\sigma_0},0,0)=0$ and 
$$
\partial_{(v,\mu)} H (u_{\sigma_0},0,0):\left\{\begin{array}{rcl}
Z \times \R \ & \to & L^2(0,2\pi), \vspace{\jot} \\
(\psi,\alpha) & \mapsto & L_{u_{\sigma_0}} \psi-\alpha u_{\sigma_0}'
\end{array}
\right.
$$
clearly defines an isomorphism due to our assumption that $u_{\sigma_0}' \notin \range L_{u_{\sigma_0}}$. By the implicit function theorem we find neighborhoods $U \subset H^2_{\text{per}}(0,2\pi)$ of $u_{\sigma_0}$, $V\subset Z$ of $0$, $J \subset \R$ of $0$ and continuously differentiable functions $v^*:U \to V$, $\mu^*:U \to J$ such that $v^*(u_{\sigma_0})=0$, $\mu^*(u_{\sigma_0})=0$ and
$$
\forall (u,v,\mu)\in U\times V\times J: \, H(u,v,\mu)=0 \Leftrightarrow v=v^*(u), \mu=\mu^*(u).
$$ 
Thus, for $|t|$ sufficiently small we find $L_{u(t)}\bigl(u_{\sigma_0}'+v^*(u(t))\bigr)=\mu^*(u(t))\bigl(u_{\sigma_0}'+v^*(u(t))\bigr)$. With $\varphi(t) \coloneqq u_{\sigma_0}'+v^*(u(t))$ and $\mu(t)\coloneqq \mu^*(u(t))$ we have $\varphi(0)=u_{\sigma_0}'$, $\mu(0)=0$ and 
\begin{equation} \label{eigenvalue}
L_{u(t)} \varphi(t)=\mu(t)\varphi(t)
\end{equation} 
so that we have found a parameterization of the eigenvalue $\mu(t)$ nearby $0$ with eigenfunction $\varphi(t)$ of $L_{u(t)}$. Next we want to compute $\dot\mu(0)$ and show that $\dot\mu(0)\neq 0$ so that the critical zero eigenvalue moves away from zero. Differentiating \eqref{eigenvalue} w.r.t. $t$ and evaluating at $t=0$ we get
\begin{equation*}
L_{u_{\sigma_0}} \dot \varphi(0)-2 \dot u(0) \overline{u_{\sigma_0}} u_{\sigma_0}'  -2 u_{\sigma_0}\overline{\dot u(0)} u_{\sigma_0}'-2 u_{\sigma_0}\dot u(0) \overline{u_{\sigma_0}'}=\dot\mu(0) u_{\sigma_0}'.
\end{equation*}
Theorem~\ref{Thm Crandall-Rabinowitz} yields $\dot v(0)=\psi_{\sigma_0}$ from which we find $\dot u(0)=-u_{\sigma_0}' \dot \sigma(0)+\psi_{\sigma_0}$. Thus, 
\begin{equation*}
L_{u_{\sigma_0}} \dot \varphi(0)-2(\psi_{\sigma_0}\overline{u_{\sigma_0}} u_{\sigma_0}'+ u_{\sigma_0}\overline{\psi_{\sigma_0}} u_{\sigma_0}'+ u_{\sigma_0}\psi_{\sigma_0}\overline{u_{\sigma_0}'})+2 \dot \sigma(0) u_{\sigma_0}'(\overline{u_{\sigma_0}} u_{\sigma_0}'+2u_{\sigma_0}\overline{u_{\sigma_0}'})=\dot\mu(0) u_{\sigma_0}'.
\end{equation*}
Using \eqref{imageofpsi'} this gives 
\begin{equation*}
L_{u_{\sigma_0}} \dot \varphi(0)-L_{u_{\sigma_0}} \psi_{\sigma_0}'-\iu e'+2\dot \sigma(0) u_{\sigma_0}'(\overline{u_{\sigma_0}} u_{\sigma_0}'+2u_{\sigma_0}\overline{u_{\sigma_0}'})=\dot\mu(0) u_{\sigma_0}'.
\end{equation*}
Testing this equation with $\phi_{\sigma_0}^*$ and using $\dot \mu(0) \in \R$ we obtain
\begin{equation*}
\RT \int_0^{2\pi} -\iu e' \overline{\phi_{\sigma_0}^*}+2 \dot \sigma(0) u_{\sigma_0}'(\overline{u_{\sigma_0}} u_{\sigma_0}'+2u_{\sigma_0}\overline{u_{\sigma_0}'})\overline{\phi_{\sigma_0}^*} \, ds=\dot\mu(0) \RT \int_0^{2\pi} u_{\sigma_0}'\overline{\phi_{\sigma_0}^*} \, ds.
\end{equation*}
Due to $u_{\sigma_0}' \notin \range L_{u_{\sigma_0}}$ we have $\RT \int_0^{2\pi} u_{\sigma_0}'\overline{\phi_{\sigma_0}^*} \, ds \neq 0$ so that
\begin{equation*}
\dot\mu(0)=\frac{\IT \int_0^{2\pi} e'(s+\sigma_0) \overline{\phi_0^*(s)} \, ds+ 2 \dot \sigma(0) \RT \int_0^{2\pi} u_0'(\overline{u_0} u_0'+2u_0\overline{u_0'})\overline{\phi_0^*} \, ds}{\RT \int_0^{2\pi} u_0'\overline{\phi_0^*} \, ds}.
\end{equation*}
From Theorem~\ref{Thm Crandall-Rabinowitz} we know that
$$
\dot\sigma(0) = -\frac{1}{2} \frac{\bigl\langle \partial^2_{(f_1,v)^2} F(\sigma_0,0,0)[(1,\psi_{\sigma_0}), (1,\psi_{\sigma_0})], \phi_{\sigma_0}^*\bigr\rangle_{L^2}}{\bigl\langle \partial^2_{(f_1,v),\sigma} F(\sigma_0,0,0)[1,\psi_{\sigma_0}], \phi_{\sigma_0}^*\bigr\rangle_{L^2}}.
$$
Therefore, using \eqref{transv_imageofpsi'} and 
$$
 \partial^2_{(f_1,v)^2} F(\sigma_0,0,0)[(1,\psi_{\sigma_0}), (1,\psi_{\sigma_0})]=-2\overline{u_{\sigma_0}} \psi_{\sigma_0}^2-4 u_{\sigma_0} |\psi_{\sigma_0}|^2
$$
we find that the condition $\dot \mu(0)\neq 0$ amounts to assumption \eqref{further_cond} of the theorem.

Finally, employing some arguments from spectral theory, we ensure that no other eigenvalue runs into zero. For $u=u_1+\iu u_2\in H^2_{\text{per}}(0,2\pi)$ let us define the $\C$-linear operator 
$$
L^{\C}_u:\left\{\begin{array}{rcl}
H^2_{\text{per}}((0,2\pi),\C^2) \ & \to & L^2((0,2\pi),\C^2), \vspace{\jot} \\
 \begin{pmatrix} \varphi_1 \\ \varphi_2 \end{pmatrix}  & \mapsto &\begin{pmatrix} -d\varphi_1''-\omega \varphi_2'+\zeta \varphi_1+\varphi_2-3u_1^2 \varphi_1-u_2^2 \varphi_1-2u_1u_2\varphi_2 \\ -d \varphi_2''+\omega \varphi_1'+\zeta \varphi_2-\varphi_1-u_1^2\varphi_2-3u_2^2\varphi_2-2u_1u_2\varphi_1 \end{pmatrix}
\end{array}
\right.
$$
which is constructed in such a way that 
$$
L^{\C}_u  \begin{pmatrix} \varphi_1 \\ \varphi_2 \end{pmatrix}= \begin{pmatrix} \RT L_u(\varphi_1+\iu \varphi_2)  \\ \IT L_u(\varphi_1+\iu \varphi_2)  \end{pmatrix}
$$  
whenever $\varphi_1,\varphi_2\in H^2_{\text{per}}((0,2\pi),\R)$. Since $L^{\C}_u$ is an index-zero Fredholm operator, its spectrum consists of eigenvalues. The real part of these eigenvalues (weighted with $\sign(d)$) is bounded from below by $c\in \R$ which is chosen such that
$$
\RT \biggl\langle \sign(d) L_u^\C  \begin{pmatrix} \varphi_1 \\ \varphi_2 \end{pmatrix},\begin{pmatrix} \varphi_1 \\ \varphi_2 \end{pmatrix} \biggr\rangle_{L^2((0,2\pi),\C^2)}  \geq c \left\|\begin{pmatrix} \varphi_1 \\ \varphi_2 \end{pmatrix} \right\|_{L^2((0,2\pi),\C^2)}^2
$$ 
holds. This implies that the resolvent set $\rho(L^{\C}_u)$ is non-empty and the compact embedding $H^2_{\text{per}}((0,2\pi),\C^2) \hookrightarrow L^2((0,2\pi),\C^2)$ ensures that $L^{\C}_u$ has compact resolvent so that $\sigma(L^{\C}_u)$ consists of isolated eigenvalues. Now choose $\varepsilon>0$ such that $\sigma(L_{u(0)}^\C) \cap \overline{B_{\varepsilon}^\C(0)}=\{0\}$. Using \cite[Chapter Four, Theorem 3.18]{Kato:101545} we find that $\sigma(L^{\C}_{u(t)}) \cap B_{\varepsilon}^\C(0)$ exactly consists of one algebraically simple eigenvalue if $|t|$ is sufficiently small. If in addition $|t|$ is chosen so small that $\mu(t) \in (-\varepsilon,\varepsilon)$ then this means $\sigma(L^{\C}_{u(t)}) \cap B_{\varepsilon}^\C(0)=\{\mu(t)\}$. But from $\dot \mu(0)\neq 0$ we know that $\mu(t)\neq 0$ for small $|t|> 0$ which guarantees that $0\notin\sigma(L_{u(t)}^\C)$ for $0<|t|<\delta^*$ and $\delta^*$ sufficiently small. Finally, $L_{u(t)}$ inherits the invertibility of $L_{u(t)}^\C$.

\medskip

\noindent
Step 3. Using $\dot{f_1}(0)=1$ and Step 2 we find a local reparameterization $(\tilde f_1,u(\tilde f_1))$ of $C(t)=(f_1(t),u(t))$ such that $L_{u(\tilde f_1)}$ is invertible for $0<\tilde f_1<f_1^*$.	Next we construct the connected set $\mathcal{C}^+_*$. For this we want to apply Theorem~\ref{GCT} to the map $T:\R \times H^1_{\text{per}}(0,2\pi) \to H^1_{\text{per}}(0,2\pi)$ from the proof of Theorem~\ref{Hauptresultat}. Note that this theorem can not be applied directly at the point $(0,u_{\sigma_0})$ since $\partial_u T(0,u_{\sigma_0})$ is not invertible. Instead, we apply it to the points $(\tilde f_1,u(\tilde f_1))$ with $\tilde f_1\in (0,f_1^*)$ and obtain that the maximal continuum $\mathcal{C}^+(\tilde f_1)\subset [\tilde f_1,\infty)\times H^1_{\text{per}}(0,2\pi)$ of solutions of \eqref{TWE} with $(\tilde  f_1, u(\tilde f_1)) \in \mathcal{C}^+(\tilde f_1)$ is unbounded or returns to another solution $u^+(\tilde f_1)\neq u(\tilde f_1)$ at $f_1=\tilde f_1$. As in the proof of Theorem~\ref{Hauptresultat} we see that the continuum $\mathcal{C}^+(\tilde f_1)$ persists as a connected and closed set in $[\tilde f_1,\infty)\times H^2_{\text{per}}(0,2\pi)$. Let us define
$$
\mathcal{C}_*^+ \coloneqq \bigcup_{\tilde f_1 \in (0,f_1^*)} \mathcal{C}^+(\tilde f_1)\subset \mathcal{C}^+.
$$
Clearly, $\operatorname{pr}_1(\mathcal{C}^{+}_*)\subset (0,\infty)$ and $\mathcal{C}^+_*$ is connected since $\mathcal{C}^+(\tilde f_1) \subset \mathcal{C}^+(\bar f_1)$ for $\bar f_1<\tilde f_1$. Let us now suppose that $\operatorname{pr}_1(\mathcal{C^+_*})\neq (0,\infty)$ so that $\operatorname{pr}_1(\mathcal{C}^+_*)$ is bounded. By (ii) this implies that $\mathcal{C}^+_*$ is bounded too. Hence $\mathcal{C}^+(\tilde f_1)$ is bounded for $\tilde f_1 \in (0,f_1^*)$ and contains the additional element $(\tilde f_1,u^+(\tilde f_1))$. Let us take $\tilde f_1= \frac{1}{n}$ and consider the two sequences of solutions $(\frac{1}{n}, u(\frac{1}{n}))_{n}$ and $(\frac{1}{n}, u^+(\frac{1}{n}))_{n}$. Using Theorem~\ref{a-priori} we obtain uniform $C^3$-bounds for both sequences $(u(\frac{1}{n}))_{n}$ and $(u^+(\frac{1}{n}))_{n}$. Therefore we can take convergent subsequences (denoted by the same index) and obtain $u(\frac{1}{n})\to u_{\sigma_0}$ and $u^+(\frac{1}{n})\to u_0^+$ in $C^2([0,2\pi])$ as $n\to \infty$. In particular $(0,u_{\sigma_0}), (0,u_0^+)\in \overline{\mathcal{C^+_*}}$ and the uniqueness property from (i) guarantees that $u_0^+\neq u_{\sigma_0}$. This finishes the proof.
\end{proof}

\begin{proof}[Proof of Corollary~\ref{Fortsetzung_nichttrivial_Korollar}.]
We first check assumption \eqref{eq:sigma_0} of Theorem~\ref{Fortsetzung_nichttrivial}. For $e(s)=\eu^{\iu k_1 s}$ we have
\begin{align*}
\IT \int_0^{2\pi} e(s+\sigma_0) \overline{\phi_0^*(s)} \,ds &= \IT\int_0^{2\pi} \eu^{\iu k_1 (s+\sigma_0)} \overline{\phi_0^*(s)} \,ds \\
&= \cos(k_1\sigma_0) \IT \int_0^{2\pi} \eu^{\iu k_1 s} \overline{\phi_0^*(s)} \, ds+\sin(k_1\sigma_0)\RT \int_0^{2\pi} \eu^{\iu k_1 s} \overline{\phi_0^*(s)} \, ds,
\end{align*}
where
\begin{align*}
\IT \int_0^{2\pi} \eu^{\iu k_1 s} \overline{\phi_0^*(s)} \, ds&= \int_0^{2\pi}  \sin(k_1 s) \RT \phi_0^*(s)-\cos(k_1 s)\IT \phi_0^*(s)  \, ds, \\
\RT \int_0^{2\pi} \eu^{\iu k_1 s} \overline{\phi_0^*(s)} \, ds&=\int_0^{2\pi}  \cos(k_1 s) \RT \phi_0^*(s) +\sin(k_1 s) \IT \phi_0^*(s) \, ds.
\end{align*}
Since assumption \eqref{extra} guarantees that $\IT \int_0^{2\pi} \eu^{\iu k_1 s} \overline{\phi_0^*(s)} \, ds$ and $\RT \int_0^{2\pi} \eu^{\iu k_1 s} \overline{\phi_0^*(s)} \, ds$ do not vanish simultaneously condition \eqref{eq:sigma_0_Korollar} ensures that assumption \eqref{eq:sigma_0} of Theorem~\ref{Fortsetzung_nichttrivial} is fulfilled. 

Next we check that assumption \eqref{eq:transv} of Theorem~\ref{Fortsetzung_nichttrivial} holds. For this we compute
\begin{equation} \label{expression}
\IT \int_0^{2\pi}e'(s+\sigma_0)\overline{\phi_0^\ast(s)}\,ds =\IT \int_0^{2\pi}\iu k_1 \eu^{\iu k_1(s+\sigma_0)}\overline{\phi_0^\ast(s)}\,ds = k_1 \RT \int_0^{2\pi}\eu^{\iu k_1(s+\sigma_0)}\overline{\phi_0^\ast(s)}\,ds.
\end{equation}
From \eqref{extra} we know that $\int_0^{2\pi}\eu^{\iu k_1(s+\sigma_0)}\overline{\phi_0^\ast(s)}\,ds = \eu^{\iu k_1\sigma_0} \int_0^{2\pi}\eu^{\iu k_1s}\overline{\phi_0^\ast(s)}\,ds  \neq 0$ and moreover $\IT \int_0^{2\pi}\eu^{\iu k_1(s+\sigma_0)}\overline{\phi_0^\ast(s)}\,ds=0$ by the definition of $\sigma_0$. Therefore the expression in \eqref{expression} does not vanish and so assumption \eqref{eq:transv} of Theorem~\ref{Fortsetzung_nichttrivial} holds. This is all we had to show. 
\end{proof}

\begin{proof}[Proof of Theorem~\ref{second_derivative}] Let us fix all parameters $d, \omega, \zeta, k_1$ and $f_0$ and consider $u: f_1\mapsto u(f_1)$ as a function mapping the parameter $f_1\in [-f_1^\ast,f_1^\ast]$ to the uniquely defined solution of \eqref{TME} in the neighborhood of the trivial solution $u_0$. The existence of such a smooth function follows from the implicit function theorem  applied to the equation $T(f_1,u)=0$, cf. proof of Theorem~\ref{Hauptresultat}. Similarly we consider the functions $v: f_1\mapsto \frac{d u(f_1)}{df_1}$ and $w: f_1\mapsto \frac{d^2 u(f_1)}{df_1^2}$. Then 
\begin{equation} \label{derivatives}
\frac{d}{df_1} \|u(f_1)\|_2^2 = 2\int_0^{2\pi} \RT(u\overline v)\,ds, \qquad \frac{d^2}{df_1^2} \|u(f_1)\|_2^2 = 2\int_0^{2\pi} \RT(u\overline w)+|v|^2\,ds
\end{equation}
and the differential equations for $v, w$ at $f_1=0$ are given by
\begin{eqnarray}
-d v'' + \mathrm{i}\omega v' + (\zeta-\mathrm{i})v - 2|u_0|^2 v-u_0^2\overline{v}+\mathrm{i}\mathrm{e}^{\mathrm{i}k_1 s}&=&0, \label{eq_v}\\
-d w'' + \mathrm{i}\omega w' + (\zeta-\mathrm{i})w - 4u_0|v|^2-2\overline{u_0} v^2-2|u_0|^2 w-u_0^2\overline{w}&=&0 \label{eq_w}
\end{eqnarray}
both equipped with $2\pi$-periodic boundary conditions. The first equation \eqref{eq_v} has a unique solution since the homogeneous equation has a trivial kernel, cf. proof of Theorem~\ref{Hauptresultat}. Thus $v(s)=\alpha \eu^{\iu k_1 s}+\beta \eu^{-\iu k_1s}$ where $\alpha, \beta\in\C$ solve the linear system 
\begin{align*}
(dk_{1}^2-k_{1}\omega+\zeta-\iu-2|u_0|^2)\alpha - u_0^2\overline\beta +\iu &= 0,\\
(dk_{1}^2+k_{1}\omega+\zeta-\iu-2|u_0|^2)\beta - u_0^2\overline\alpha &=0.
\end{align*}
Solving for $\alpha, \beta$ leads to the formulae in the statement of the theorem. Since $v$ is the sum of two $2\pi$-periodic complex exponentials and $u_0$ is a constant we see from \eqref{derivatives} that $\frac{d}{df_1} \|u(f_1)\|_2^2\mid_{f_1=0} =0$. Having determined $v$ we can consider the second equation \eqref{eq_w} as an inhomogeneous equation for $w$. It also has a unique solution since the homogeneous equation is the same as in \eqref{eq_v}. Since the inhomogeneity is of the form $c_1 \eu^{\iu 2k_1s}+c_2 \eu^{-\iu 2k_1 s}+c_3$ the solution has the form $w(s) = \gamma \eu^{\iu 2k_1 s}+\delta \eu^{-\iu 2 k_1s}+\epsilon$. Moreover, for the determination of $\frac{d^2}{df_1^2} \|u(f_1)\|_2^2$ the values of $\gamma, \delta$ are irrelevant and only the value of $\epsilon$ matters. Using 
$$
|v|^2 = |\alpha|^2+|\beta|^2+2\RT(\alpha\overline\beta \eu^{\iu 2 k_1 s}), \quad v^2 = \alpha^2 \eu^{\iu 2k_1 s}+\beta^2 \eu^{-\iu 2k_1 s}+2\alpha\beta
$$
we find from \eqref{eq_w} that the equation determining $\epsilon$ is 
$$
(\zeta-\iu)\epsilon -4u_0(|\alpha|^2+|\beta|^2)-4\overline{u_0} \alpha\beta-2 |u_0|^2\epsilon -u_0^2 \overline\epsilon=0.
$$
Since this is an equation of the form $x\epsilon + y\overline \epsilon=z$ with $x, y, z$ given in the statement of the theorem we find the solution formula $\epsilon=\frac{-\overline z y+z\overline x}{|x|^2-|y|^2}$. Finally, only the constant contributions from $\overline w$ and $|v|^2$ contribute to the integral in the formula \eqref{derivatives} for $ \frac{d^2}{df_1^2} \|u(f_1)\|_2^2$ and lead to the claimed statement of the theorem.
\end{proof}

\section*{Appendix}~

Here we raise the issue mentioned in Remark~\ref{Fortsetzung_nichttrivial_Remark}.($\gamma$) that assumption \eqref{extra} from Corollary~\ref{Fortsetzung_nichttrivial_Korollar} is not satisfied if $u_0$ is $\tfrac{2\pi}{j}$-periodic and $j\in\N$ is not a divisor of $k_1$. Let us first prove that $\phi_0^*$ (spanning $\ker L_{u_0}^*$) inherits several properties from $u_0'$ (spanning $\ker L_{u_0}$).

\begin{Proposition} \label{parity} Let $u_0\in H^2_{\text{per}}(0,2\pi)$ be a non-constant non-degenerate solution of \eqref{TWE} for $f_1=0$ and let $\ker L^*_{u_0}=\spann\{\phi_0^*\}$. Then the following holds:
\begin{itemize}
\item[(i)] If $u_0$ is $\frac{2\pi}{j}$-periodic with $j\in \N$ then $\phi_0^*$ is $\frac{2\pi}{j}$-periodic.
\item[(ii)] If $\omega=0$ and if $u_0$ is even then $\phi_0^*$ is odd.
\end{itemize}
\end{Proposition}

\begin{proof}
(i) By assumption we have that $\ker L_{u_0} = \spann\{u_0'\}$ and $u_0'$ is a $\frac{2\pi}{j}$-periodic function. Let us define $D \coloneqq \{\varphi\in H_{\text{per}}^2(0,2\pi):\varphi \text{ is $\frac{2\pi}{j}$-periodic}\}$ and similarly $L^2_{j}(0,2\pi)=\{\varphi\in L^2(0,2\pi):\varphi \text{ is $\frac{2\pi}{j}$-periodic}\}$. If we consider the restriction
$$
L_{u_0}^\#:\left\{\begin{array}{rcl}
D & \to & L^2_j(0,2\pi), \vspace{\jot} \\
\varphi & \mapsto & L_{u_0} \varphi,
\end{array}
\right.
$$
then $L_{u_0}^\#$ is again an index-zero Fredholm operator with $\ker L_{u_0}^\#= \spann\{u_0'\}$. Further we have $(L_{u_0}^\#)^* = (L_{u_0}^*)^\#$ where
$$
(L_{u_0}^*)^\#:\left\{\begin{array}{rcl}
D & \to & L^2_j(0,2\pi), \vspace{\jot} \\
\varphi & \mapsto & L_{u_0}^* \varphi
\end{array}
\right.
$$ 
is the restriction of the adjoint.
But since $1=\dim \ker (L_{u_0}^*)^\# = \dim \ker L_{u_0}^*$ it follows that $\ker (L_{u_0}^*)^\# = \ker L_{u_0}^*$ and hence $\phi_0^* \in D$ as claimed.

\medskip
 
The proof of (ii) is very similar. Due to the assumption $\omega=0$ we can restrict both the domain and the codomain of $L_{u_0}$ to odd functions and observe that it is still an index-zero Fredholm operator.
\end{proof}

Instead of $k_1\in \N$ let us consider a perturbation $k_1(\epsilon)\in \R\setminus\{k_1\}$ with $\lim_{\epsilon\to 0} k_1(\epsilon)=k_1$. For $\epsilon\approx 0$ one may have maximally connected continua ${\mathcal C}_\epsilon^+$ as described in Theorem~\ref{Fortsetzung_nichttrivial}. In a topological sense one can describe $\liminf \{\mathcal C_{\epsilon}^+:\epsilon^{-1}\in\N\}$ and $\limsup \{\mathcal C_{\epsilon}^+:\epsilon^{-1}\in\N\}$ as in \cite{Whyburn}. However, having in mind sequences of loops degenerating to one point, we do not intend to make any existence statement about a bifurcating branch obtained through such a topological limiting procedure. Let us abbreviate by $e_{\epsilon}(s)$ the periodic extension of $[0,2\pi)\to\C, \, s\mapsto \eu^{\iu k_1(\epsilon)s}$ onto $\R$. Note that
\begin{align*}
&\hspace{0.2cm}\IT \int_0^{2\pi} e_{\epsilon} (s+\sigma_{0,\epsilon}) \overline{\phi_0^*(s)} \, ds =\IT \int_0^{2\pi} \eu^{\iu k_1(\epsilon)s} \overline{\phi_{\sigma_{0,\epsilon}}^*(s)} \, ds = \IT \int_{-\sigma_{0,\epsilon}}^{2\pi-\sigma_{0,\epsilon}} \eu^{\iu k_1(\epsilon)(s+\sigma_{0,\epsilon})} \overline{\phi_0^*(s)} \, ds \\
&=\cos(k_1(\epsilon)\sigma_{0,\epsilon}) \IT \int_{-\sigma_{0,\epsilon}}^{2\pi-\sigma_{0,\epsilon}} \eu^{\iu k_1(\epsilon)s} \overline{\phi_0^*(s)} \, ds+\sin(k_1(\epsilon)\sigma_{0,\epsilon}) \RT \int_{-\sigma_{0,\epsilon}}^{2\pi-\sigma_{0,\epsilon}} \eu^{\iu k_1(\epsilon)s} \overline{\phi_0^*(s)} \, ds
\end{align*}
so that assumption \eqref{eq:sigma_0} from Theorem~\ref{Fortsetzung_nichttrivial} becomes
\begin{equation*}
\tan(k_1(\epsilon)\sigma_{0,\epsilon}) =\frac{\int_{-\sigma_{0,\epsilon}}^{2\pi-\sigma_{0,\epsilon}} \cos(k_1(\epsilon)s) \IT \phi_0^*(s)  -\sin(k_1(\epsilon)s) \RT \phi_0^*(s) \, ds}{\int_{-\sigma_{0,\epsilon}}^{2\pi-\sigma_{0,\epsilon}} \sin(k_1(\epsilon)s) \IT \phi_0^*(s)+ \cos(k_1(\epsilon)s) \RT \phi_0^*(s) \, ds}. 
\end{equation*}
One may expect that if (as a result of such a limiting procedure) a bifurcating branch at $k_1=\lim_{\epsilon\to 0} k_1(\epsilon)$ exists then it bifurcates at $\sigma_0=\lim_{\epsilon\to 0} \sigma_{0,\epsilon}$ determined from
\begin{align*}
\tan(k_1\sigma_0) &= \lim_{\epsilon\to 0} \frac{\int_{-\sigma_{0,\epsilon}}^{2\pi-\sigma_{0,\epsilon}} \cos(k_1(\epsilon)s) \IT \phi_0^*(s)  -\sin(k_1(\epsilon)s) \RT \phi_0^*(s) \, ds}{\int_{-\sigma_{0,\epsilon}}^{2\pi-\sigma_{0,\epsilon}} \sin(k_1(\epsilon)s) \IT \phi_0^*(s)+ \cos(k_1(\epsilon)s) \RT \phi_0^*(s) \, ds}\\
& =\frac{\int_{-\sigma_{0}}^{2\pi-\sigma_{0}} s\sin(k_1 s)\IT \phi_0^*(s)+s \cos(k_1 s)\RT \phi_0^*(s)\,ds}{\int_{-\sigma_{0}}^{2\pi-\sigma_{0}} s\sin(k_1 s)\RT \phi_0^*(s)-s \cos(k_1 s)\IT \phi_0^*(s)\,ds}.
\end{align*}
However, this is not supported by our numerical experiments and we have to leave the correct determination of $\sigma_0$ in this case as an open question.

\section*{Acknowledgments} Funded by the Deutsche Forschungsgemeinschaft (DFG, German Research Foundation) -- Project-ID 258734477 -- SFB 1173.

\bibliographystyle{plainurl}	
\bibliography{bibliography}
	
\end{document}